\documentclass{article}
\usepackage[left=2cm,right=2cm,top=2cm,bottom=2cm]{geometry}
\usepackage{lineno,hyperref}
\modulolinenumbers[5]
\usepackage{amsmath}
\usepackage{amssymb}
\usepackage{float}
\usepackage{tikz}
\usepackage{mathtools}
\usepackage{hyperref}
\usepackage{lipsum} 
\usepackage{commath}
\usepackage{cases}
\usepackage{amsthm}
\usepackage{ifpdf}
\ifpdf
  \usepackage{auto-pst-pdf}
\else
  \usepackage{pstricks}
  \fi
  
\newcommand{\Rref}{\ref}

  \usepackage[sc,osf]{mathpazo}
   \newcommand{\yy}{\mathbf{y}}
 \newcommand{\eee}{\mathbf{e}}

\newcommand{\ttens}{\mathcal{T}}

\newcommand{\dd}{\mathcal{D}}
\newcommand{\ff}{\mathcal{F}}
\newcommand{\uu}{u}
\newcommand{\nn}{\textbf{n}}
\newcommand{\xx}{\textbf{x}}

\newcommand{\vv}{v}

\newcommand{\mm}{\mathcal{M}}

\newcommand{\llm}{\mathcal{L}}
\newcommand{\pt}{\partial}

\usepackage{graphicx} 
\usepackage{listings}
\definecolor{darkWhite}{rgb}{0.94,0.94,0.94}

\newtheorem{thrm}{Theorem}[section]
\newtheorem{prpstn}[thrm]{Proposition}
\newtheorem{dfntn}[thrm]{Definition}
\newtheorem{rmrk}[thrm]{Remark}
 
\begin{document}

  \title{\noindent Error estimate of the Non Intrusive Reduced Basis method with finite volume schemes }
\maketitle
\normalsize
\begin{center}
\author{Elise Grosjean \footnotemark[1],}
\author{Yvon Maday \footnotemark[1] \footnotemark[2]}
\end{center}

\footnotetext[1]{Sorbonne Universit\'e and Universit\'e de Paris, CNRS, Laboratoire Jacques-Louis Lions (LJLL), F-75005 Paris, France}
\footnotetext[2]{Institut Universitaire de France}

\date

%\maketitle

\begin{abstract}
%\textbf{Error estimate of the Non Intrusive Reduced Basis method with finite volume schemes}\\
The context of this paper is the simulation of parameter-dependent partial differential equations (PDEs). When the aim is to solve such PDEs for a large number of parameter values, Reduced Basis Methods (RBM) are often used to reduce computational costs of a classical high fidelity code based on Finite Element Method (FEM), Finite Volume (FVM) or Spectral methods. The efficient implementation of most of these RBM requires to modify this high fidelity code, which cannot be done, for example in an industrial context if the high fidelity code is only accessible as a "black-box" solver. The Non Intrusive Reduced Basis method (NIRB) has been introduced in the context of finite elements as a good alternative to reduce the implementation costs of these parameter-dependent problems. The method is efficient in other contexts than the FEM one, like with finite volume schemes, which are more often used in an industrial environment. In this case, some adaptations need to be done as the degrees of freedom in FV methods have different meanings. At this time, error estimates have only been studied with FEM solvers. 
In this paper, we present a generalisation of the NIRB method to Finite Volume schemes and we show that estimates established for FEM solvers also hold in the FVM setting. We first prove our results for the hybrid-Mimetic Finite Difference method (hMFD), which is part the Hybrid Mixed Mimetic methods (HMM) family. Then, we explain how these results apply more generally to other FV schemes. Some of them are specified, such as the Two Point Flux Approximation (TPFA). 
\end{abstract}
\vspace{1cm}
Keywords: Reduced Basis Method, Finite Volume Method
%\maketitle

%\linenumbers
\section*{Introduction}
This paper is concerned with the efficient simulation of parameter-dependent partial differential equations (PDEs), with a parameter varying in a given set $\mathcal{G}$. For complex physical systems, computational costs can be huge. It may happen, for instance in the context of parameter optimization or real time simulations in an industrial context, that the same problem needs to be solved for several parameter values.

In such cases, different model order reductions (MOR) like the reduced basis methods have been proposed (see eg~\cite{rb, hesthaven2016certified}) based on POD or greedy selection of the reduced basis, the reduced basis elements being computed accurately enough through a high fidelity code. In these approaches, the efficient implementation of the reduced method, leading to reductions in the computational time, requires to be able to deeply enter into the high fidelity code, in order to compute offline, a key ingredient which saves the implementation costs online. This can be tedious, even impossible when the code has been bought, as it is often the case in an industrial context. The Non Intrusive Reduced Basis methods (NIRB)~\cite{madaychakir,NIRB1} has been proposed in this framework. This method is useful to reduce computational costs of parametric-dependent PDEs in a non intrusive way.  Unlike other MOR, the NIRB method does not require to modify the solver code and hence does not depend on the numerical approach underlying the code.  %The solver can use any classical procedure.
This method, based on two grids, one fine where high fidelity computations are done offline and one coarse which is used online, has been introduced in~\cite{madaychakir,NIRB1}. It was presented and analysed in the case where the high fidelity code is based on a finite element solver. In these papers, an optimal error estimate is recovered and illustrated with numerical simulations. The method can be extended to other classical discretizations but the key ingredient is a better approximation rate in the $L^2$ norm than in the energy norm, thanks to the Aubin-Nitsche's  trick that is easy for variational approximations. In addition, the degrees of freedom in FVM don't have the same status as in FEM and the transfer of information from one grid to another must be adapted. 
The aim of this paper is to propose the adaptation of the NIRB method to FV and to propose the numerical analysis able to recover the classical error estimate with Finite Volume (FV) schemes. 

\subsection*{The Non Intrusive Reduced Basis Method.}

Let $\Omega$ be an open bounded domain in $\mathbb{R}^d$ with $d\leq 3$.
The NIRB method in the context of a high fidelity solver of finite element or finite volume types involves two partitioned meshes, one fine mesh  $\mathcal{M}_h$ and one coarse mesh $\mathcal{M}_H$, where $h$ and $H$ are the respective sizes of the meshes and $h<\! < H$. The size $h$ (res. $H$) is defined as $h=\underset{K \in \mathcal{M}_h}{\max\ } h_{K}$ (resp. $H=\underset{K \in \mathcal{M}_H}{\max\ } H_{K}$) where the diameter $h_K$ (or $H_K$) of any element $K$ in a mesh is equal to $\underset{x,y \in K}{\sup \ } |x-y|$. \\ 

As it is classical in other reduced basis methods, the NIRB method is based on the assumption (assumed or actually checked) that the manifold of all solutions $\mathcal{S} = \{u(\mu), \mu \in \mathcal{G} \}$ has a small Kolmogorov n-width $\varepsilon(n)$~\cite{kolmo}. This leads to the fact that very few well chosen solutions are sufficient to approximate well any element in~$\mathcal{S}$. These well chosen elements are called the snapshots. In this frame, the method
 is based on two steps :
one offline step and one online. The ``offline'' part is costly in time because the snapshots must be generated with a  high fidelity code on the fine mesh $\mathcal{M}_h$. The  ``online'' step 
is performed on the coarse mesh $\mathcal{M}_H$, and thus much less expensive than a high fidelity computation. This algorithm remains effective as the offline part is performed only once and in advance and also independently from the online stage. The online stage can then be done as many times  as desired.

    %    \small Two stages: one \textcolor{orange!70}{\textbf{offline}}, costly in time, and one \textcolor{blue!70}{\textbf{online}}, much faster.\\
% \centering
%\vspace{1cm}
%\begin{center}
%  \mypicture{}
%\end{center}
%\vspace{1cm}
\begin{itemize}
    \item In the offline part, several snapshots are  computed on the fine mesh for different well chosen parameters in the parameter set $\mathcal{G}$ with the (fine and costly) solver. The best way to determine the required parameters is through a greedy procedure~\cite{veroy2003posteriori, barrault2004empirical, greedy} if available or through an SVD approach. 
    \item The online part consists in computing a coarse solution with the same solver for some (new) parameter $\mu \in \mathcal{G}$ and then $L^2$-project this (coarse) solution on the (fine) reduced basis. This results in an improved approximation, in the sense that we may retrieve almost fine error estimates with a much lower computational cost. 
\end{itemize}
\subsection*{Motivation and earlier works.}
Several papers have underlined the efficiency of the NIRB method in the finite element context, illustrated both with numerical results presenting error plots and the online part compurational time~\cite{madaychakir,NIRB1, NIRB2,NIRB3}. 
However, to the best of our knowledge, works with Finite Volume (FV) schemes have not yet been studied with a non intrusive approach~\cite{iliev2013two,stabile2017pod,haasdonk2008reduced,TPFABR,Casenave2014}, and they are often preferred to finite element methods in an industrial context. Thanks to recent works on super-convergence~\cite{supconv2}, and with some technical subtleties, we are now able to generalize the two-grids method which is non intrusive to FV methods and propose the numerical analysis of this method.

\begin{center}
\begin{tikzpicture}[every text node part/.style={align=center}]
      %  \tikzstyle{noeud}=[minimum width=5cm,minimum height=3cm,rectangle,rounded corners=10pt,draw,fill=yellow!75,text=red,font=\bfseries]
      
     \node[draw=black,thick,rectangle,rounded corners=3pt,fill=orange!40] (snap)at(0,4.5){   \textcolor{black}{ Snapshots:} \\  \textcolor{black}{\{$u_h(\mu_1),\ldots,u_h(\mu_{N})$\}} \\ \textcolor{black}{computed on a fine mesh $\mathcal{M}_h$}};

    \draw[color=orange] (-5,2) rectangle (5,6);
    \draw (0,5.7) node[] {\textbf{Offline}};
     
 \node[draw,rectangle,color=blue!80,minimum width=10cm,minimum height=4cm](aaaa)at(0,-0.5){};
     %\draw[color=blue!80] (-4,-1.5) rectangle (4,1.5);
 \draw (0,1.2) node[] {\textbf{Online}};
        \node[draw,circle,rounded corners=2pt,fill=black!25] (solver)at(-5,1.6){\quad   Solver   \quad};
        \node[draw,rectangle,rounded corners=3pt,fill=orange!40] (basis)at(0,2.8){Orthonormal basis: $(\Phi_i^h)_{i=1,\ldots,N}$};
        \node[draw=yellow,thick,rectangle,rounded corners=3pt,fill=blue!30](coarsesol)at(0,0.3){Coarse solution: $u_H(\mu)$ computed on $\mathcal{M}_H$};
    
         \node[draw,rectangle,rounded corners=3pt,fill=blue!25] (app)at(0,-1.8){\textcolor{red}{NIRB approximation: $u_{hH}^N(\mu)$} };
         
          \draw[->,draw=orange,fill=blue,line width=0.3mm] (snap) -- (basis) node[midway,left]{\textcolor{orange}{Greedy}};
         \draw[->,draw=orange,fill=blue,line width=0.3mm] (basis) -- (aaaa){};
   
          \draw[->,draw=black,fill=blue,line width=0.3mm] (solver) -- (snap) node[midway,left]{};
         \draw[->,draw=black,fill=blue,line width=0.3mm] (solver) -- (coarsesol) node[midway,left]{};
  \draw[->,draw=blue!80,fill=blue,line width=0.3mm] (coarsesol) -- (app) node[midway,left]{\textcolor{blue!80}{$L^2$-projection: set $u_{hH}^N(\mu) =$}\\ \textcolor{blue!80}{ $\overset{N}{\underset{i=1}{\sum}}(u_H(\mu),\Phi_i^h)\Phi_i^h$}};
   %     \draw[->,draw=blue,fill=blue,line width=0.3mm] -- (basis) node[midway,left]{};
    %    \draw[->,draw=blue,fill=blue,line width=0.3mm] (basis) --  node[midway,right]{};

     %   \draw[->,draw=blue,fill=blue,line width=0.3mm]  -- (app) node[midway,right]{};

\end{tikzpicture}

  \end{center}

\subsection*{Main results.}

In the context of $P_1$-FEM solvers, the works~\cite{madaychakir,NIRB1} retrieve an estimate error of the order of $\mathcal{O}(h+H^2)$ in the energy norm using the Aubin-Nitsche's Lemma \cite{FE} for the coarse grids solution  (for a reduced basis dimension large enough). With FV schemes, no equivalent of the Aubin-Nitsche's lemma is available, instead, we consider the class of Hybrid Mimetic Mixed methods (HMM) schemes for elliptic equations and use a super-convergence property proven in (~\cite{supconv1,supconv2,tpfascheme}).\\
Let us consider the following linear second-order parameter dependent problem as our model problem:
\begin{subnumcases}
     \strut - \ \textrm{div} (A(\mu) \nabla u)=f \textrm{ in } \Omega,& \label{ellip1}\\
    \label{ellip2} u = 0 \textrm{ on } \pt \Omega, &
\end{subnumcases}
where  $f\in L^2(\Omega)$,  $\mu$ is  a parameter in a set $\mathcal{G}$, and for any $\mu \in \mathcal{G}$,  $A(.; \mu):\Omega \to
\mathbb{R}^{d\times d}$ is measurable, bounded, uniformly elliptic, and $A(\xx; \mu)$ is symmetric for a.e. $\xx \in \Omega$.

%There exist $a_-,a_+ > 0$ such that, for a.e. $ \xx \in \Omega, \forall \xi \in \mathbb{R}^2, \ a_-|\xi|^2 \leq A(\xx) \xi\cdot \xi \leq a_+ |\xi|^2.$\\ 
Under general hypotheses, it is well known that this problem has a unique solution.

\noindent The usual weak formulation for problem \eqref{ellip1}-\eqref{ellip2} reads:\\
Find $u \in H_0^1(\Omega)$ such that,
\begin{equation}
    \forall v \in H_0^1(\Omega), \quad a(u,v;\mu)=(f,v),
    \label{varellip}
\end{equation}
where \begin{equation*}
    a(w,v;\mu)=\int_{\Omega}A(\xx;\mu)\nabla w(\xx)\cdot \nabla v (\xx)\ d\xx,\quad \forall w, v\in H_0^1(\Omega).
\end{equation*}

The main result of this paper is the following estimate:

\begin{thrm} [NIRB error estimate for hMFD solvers]
       \label{th11}
   Let $u_{hH}^N(\mu)$ be the reduced solution projected on the fine mesh and generated with the hMFD solver with the unknowns defined on $\xx_k=\overline{\xx}_K$ (the cell centers of mass), and $u(\mu)$ be the exact solution of \eqref{varellip} under an $H^2$ regularity assumption \eqref{h2reg} (which will be stated later), then the following estimate holds \\
\begin{equation}
    \norm{u(\mu) - u_{hH}^N(\mu)}_{\dd}\leq \varepsilon(N) +C_1 h + C_2(N) H^2, 
    \label{estimationNIRB}
\end{equation}
where $C_1$ and $C_2$ are constants independent of $h$ and $H$,$C_2$ depends on $N$, the number of functions in the basis, and $\norm{\cdot}_{\dd}$ is the discrete norm introduced in section \Rref{sect2}, and $\varepsilon$ depends of the Kolmogorov n-width. If $H$ is such as $H^2\sim h$, and $\varepsilon(N)$ small enough, it results in an error estimate in $\mathcal{O}(h)$.
\end{thrm}
Note that if $H$ is chosen such as $H^2 \sim h$ and $\varepsilon(N)$ small enough, it results in an error estimate in $\mathcal{O}(h)$. 

\subsection*{Outline  of  the  paper.} The  rest  of  this  paper  is  organized  as  follows. In section \Rref{sect2} we describe the mathematical context. In section \Rref{sect3} we recall the two-grids method. Section \Rref{sect4} is devoted to the proof of theorem \Rref{th11} with the hybrid-Mimetic Finite Difference scheme (hMFD). Section \Rref{sect5} generalizes theorem \Rref{th11} to other schemes, such as the Two Point Flux Approximation (TPFA). In the last section, the implementation is discussed and we illustrate the estimate with several numerical results on the NIRB method. 

\section{Mathematical Background}
\label{sect2}
\subsection{The Hybrid Mimetic Finite Difference method (hMFD)}

In this section, we recall the hybrid-Mimetic Finite Difference method (hMFD)~\cite{hmfd} and all the notations that will be necessary for the analysis of NIRB method in this finite volume context. 

This scheme uses interface values and fluxes as unknowns. The hMFD scheme, which is part of the family of Hybrid Mimetic Mixed methods (HMM) (\cite{cindy,hmm2,jd,hmm3,supconv1}), is a finite volume method despite its name. Indeed hMFD scheme relies on both a flux balance equation and on a local conservativity of numerical fluxes. HMM also includes mixed finite volume schemes (MFV) \cite{mfv} and hybrid finite volume schemes (HFV), a hybrid version of the SUSHI scheme~\cite{hfv}. This scheme is built on a general mesh, namely a polytopal mesh, which is a star-shaped mesh regarding the unknowns of the cells. \\

\noindent Describing the hMFD method requires to introduce the Gradient Discretisation (GD) method~\cite{cindy}, which is a general framework for the definition and the convergence analysis of many numerical methods (finite element, finite volume, mimetic finite difference methods, etc).\\
The GD schemes involve a discete space, a reconstruction operator and a gradient operator, which taken together are called a Gradient Discretisation. Selecting the gradient discretisation mostly depends on the boundary conditions (BCs). We now introduce the definition of GD for Dirichlet BCs as in~\cite{cindy} and the GD scheme associated to our model problem. \\

%$S$ can take be any set of variable parameters such that $A(\xx;\mu)$ satisfies the ellipticity condition i.e. if for some $r>0$, the diffusion coefficient $A$ is such that $A(\xx;\mu)\geq r, \ \forall \xx \in \Omega, \ \forall \mu \in S.$ Therefore, a typical parameter range is a set $S\subset\{\mu \in \mathbb{R}^n; \ A(\mu) \in L^\infty(\Omega),\ A(\xx, \mu)\geq r\  \forall \xx \in \Omega \}$, which in addition is assumed to be compact in $L^{\infty}(\Omega)$.
%If we consider an affine representation of $S$, which exists for any compact set $S$ in a Banach space (\cite{cohen}), we can write the diffusion coefficient as \begin{equation}
%    A(\xx; \mu)=\overline{A}(\xx)+\overset{n}{\underset{j=1}{\sum}} \mu_j \zeta_j(\xx), \quad \mu=(\mu_j)_{j=1,\cdots, n}\in \mathbb{R}^n, 
%\end{equation}
%where $\overline{A}$ is a fixed function from $L^{\infty}(\Omega)$, and the functions $(\zeta_j)_{j=1,\cdots,n}$ are the affine representers of $S$.

\begin{dfntn}(Gradient Discretisation) For homogeneous
Dirichlet BCs, a gradient discretisation $\dd$ is a triplet $(X_{\dd,0},\Pi_{\dd},\nabla_{\dd})$, where the space of degrees of
freedom $X_{\dd,0}$ is a discrete version of the continuous space $H_0^1(\Omega)$.
\begin{itemize}
    \item $\Pi_{\dd}:X_{\dd,0}\to L^2(\Omega)$ is a function reconstruction operator that relates an element of $X_{\dd,0}$ to a function in $L^2(\Omega)$.
\item $\nabla_{\dd}:X_{\dd,0}\to L^2(\Omega)^d$ is a gradient reconstruction in $L^2(\Omega)$ from the degrees of freedom. It must be chosen such that $\norm{\cdot}_{\dd}=\norm{\nabla_{\dd} \cdot}_{{L^2(\Omega)}^d}$ is a norm on $X_{\dd,0}$.
\end{itemize}
\end{dfntn}
\noindent In what follows, we will refer to $\Pi_\dd^H$ or $\Pi_\dd^h$ depending on the mesh considered and for the gradient reconstruction too (respectively $\nabla_\dd^H$ or $\nabla_\dd^h$).\\

\begin{dfntn} (Gradient discretisation scheme) For the variational form \eqref{varellip}, the related gradient discretisation scheme with the new operators is defined by:\\
Find $u_{\dd} \in X_{\dd,0}$ such that, $\forall v_\dd \in X_{\dd,0},$
\begin{equation}
     \int_{\Omega}A(\mu) \nabla_{\dd} u_{\dd} \cdot \nabla_{\dd} v_\dd \ d\xx=\int_{\Omega}f\  \Pi_{\dd} v_\dd \ d\xx.
    \label{variat}
\end{equation}
\end{dfntn}
We will use two general polytopal meshes (Definition 7.2~\cite{cindy}) which are admissible meshes for the hMFD scheme.

\begin{dfntn} (Polytopal mesh) Let $\Omega$ be a bounded polytopal open subset of $\mathbb{R}^d (d\geq 1).$ A polytopal mesh of $\Omega$ is a quadruplet $\mathcal{T} = (\mm, \mathcal{F},\mathcal{P}, \mathcal{V})$, where:
\begin{enumerate}
    \item $\mathcal{M}$  is a finite family of non-empty connected polytopal open disjoint subsets $\Omega$ (the cells) such that $\overline{\Omega}=\underset{K\in\mathcal{M}}{\cup}\overline{K}.$ 
    For any $K\in \mathcal{M}$,\ $|K|>0$ is the measure of $K$ and $h_K$ denotes the diameter of $K$.
\item $\mathcal{F} =  \mathcal{F}_{int}\cup  \mathcal{F}_{ext}$ is a finite family of disjoint subsets of $\overline{\Omega}$ (the edges of the mesh in 2D), such that any $\sigma \in \mathcal{F}_{int}$ is contained in $\Omega$ and any $\sigma \in \mathcal{F}_{ext}$ is contained in $\pt \Omega$. Each $\sigma \in \mathcal{F}$ is assumed to be a nonempty open subset of a hyperplane of $\mathbb{R}^d$, with a positive $(d-1)$-dimensional measure $|\sigma|$. Furthermore, for all $K \in \mathcal{M}$, there exists a subset
$\mathcal{F}_K$ of $\mathcal{F}$ such that $\pt K = \underset{\sigma \in \mathcal{F}_K}{\cup}\overline{\sigma}$. 
 We assume that for all $\sigma \in \mathcal{F}, \mathcal{M}_{\sigma}=\{K\in \mm:\ \sigma \in \mathcal{F}_K\}$ has exactly one element and $\sigma \subset \pt \Omega$ or $\mathcal{M}_{\sigma}$ has two elements and $\sigma \subset\Omega$.
%We set Mσ = {K ∈ M : σ ∈ FK} and assume that, for all σ ∈ F, either Mσ has exactly one element andthen σ ∈ Fext, or Mσ has exactly two elements and then σ ∈ Fint.
The center of mass is $\overline{\xx}_{\sigma}$, and, for $K \in \mm$ and $\sigma \in \mathcal{F}_K$, $\nn_{K,\sigma}$ is the (constant) unit vector normal to $\sigma$ outward to $K$.
\item $\mathcal{P}$ is a family of points of $\Omega$ indexed by $\mathcal{M}$ and $\mathcal{F}$, denoted by   $\mathcal{P}=((\xx_K)_{K\in \mathcal{M}},(\xx_{\sigma})_{\sigma \in \mathcal{F}} )$, such that for all $K \in \mm,\ \xx_K \in K$ and for all $\sigma \in \ff,\  \xx_{\sigma} \in \sigma$. We then denote by $d_{K,\sigma}$ the signed orthogonal distance between $\xx_K$ and $\sigma \in \ff_K$, that is:\\
$d_{K,\sigma } = (\xx - \xx_K) \cdot \nn_{K,\sigma}$, for all $\xx \in \sigma. $
We then assume that each cell $K \in \mm$ is strictly star-shaped with respect to $\xx_K$, that is $d_{K,\sigma} > 0$ for all $\sigma \in \ff_K$. This implies that for all $\xx  \in K$, the line segment $[\xx_K, \xx]$
is included in $K$. We denote $\overline{\xx}_K $ the center of mass of $K$ and by $\overline{\xx}_{\sigma} $ the one of $\sigma$.
For all $K \in \mathcal{M}$ and $\sigma \in \mathcal{F}_K$, we denote by $D_{K,\sigma}$ the cone with vertex $\xx_K$ and basis $\sigma$, that is $D_{K,\sigma}= \{t\xx_K + (1 - t)\mathbf{y}, t \in (0, 1), \mathbf{y} \in \sigma \}$. 
%We denote, for all σ ∈ F, Dσ =SK∈Mσ DK,σ (this set is called the
%“diamond” associated to the face σ, and for obvious reasons DK,σ is also
%referred to as an “half-diamond”).
\item $\mathcal{V}$ is a set of points (the vertices of the mesh). For $K \in \mathcal{M}$, the set of
vertices of $K$, i.e. the vertices contained in $\overline{K}$, is denoted $\mathcal{V}_K$. Similarly,
the set of vertices of $\sigma \in F$ is $\mathcal{V}_\sigma$.
\end{enumerate}
\end{dfntn}
%\vspace{1cm}

\begin{figure}[H]
\begin{center}
\begin{pspicture}(-1,-1)(6,6)
  \psline(0,0)(2,-0.75)(5,2.75)(3,4.5)(0.5,5)(0,0)
  \psline[linestyle=dashed](0,0)(1.8,1.7)
 % \psline[linestyle=dashed](0.75,2)(1.8,1.7)
  \psline[linestyle=dashed](0.5,5)(1.8,1.7)
  \psline[linestyle=dashed](3,4.5)(1.8,1.7)
%  \psline[linestyle=dashed](3.5,3.5)(1.8,1.7)
  \psline[linestyle=dashed](5,2.75)(1.8,1.7)
  \psline[linestyle=dashed](2,-0.75)(1.8,1.7)
\psdots[dotstyle=+,dotscale=2.5,dotangle=45](1.8,1.7)
\uput[d](1.7,1.45){$\mathbf{x}_K$}
\pspolygon[opacity=0.3,fillstyle=solid,fillcolor=yellow](1.8,1.7)(2,-0.75)(5,2.75)
\uput[d](5.1,2.75){$\sigma$}
\uput[d](2.1,4.){$K$}
\psline[linestyle=dashed]{<->}(1.8,1.7)(3.15,0.67)
\uput[d](2.6,1.1){$d_{K,\sigma}$}
\uput[d](3.3,1.8){$D_{K,\sigma}$}
\psline{->}(4.15,1.7)(4.6,1.4)
\uput[d](4.65,2.1){$\mathbf{n}_{K,\sigma}$}
\end{pspicture}
\end{center}
\caption{A cell $K$ of a polytopal 2D mesh}
\end{figure}
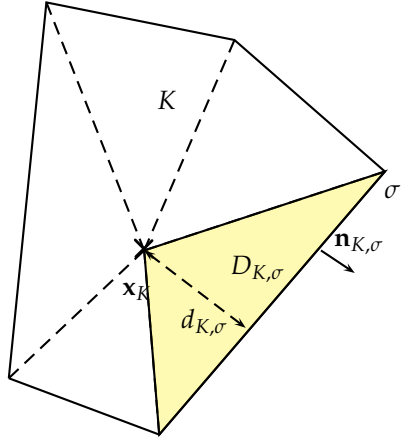

The regularity factor for the mesh is
\begin{equation}
  \theta=\underset{\sigma \in \ff_{int},\mm_{\sigma=\{K,K'\}}}{\max }
  \frac{d_{K,\sigma}}{d_{K',\sigma}}+\underset{K\in \mm}{\max}(\underset{\sigma\in \ff_K}{\max}\frac{h_K}{d_{K,\sigma}}+\textrm{Card}(\ff_K)).\label{reg}
\end{equation}

In what follows, we will consider two polytopal meshes. The fine mesh will be denoted $\mathcal{T}^h= (\mathcal{M}^h,\ff^h,\mathcal{P}^h,\mathcal{V}^h)$ and $\mathcal{T}^H= (\mathcal{M}^H,\ff^H,\mathcal{P}^H,\mathcal{V}^H)$ will be referred to as the coarse mesh.\\

\noindent All HMM schemes require to choose one point inside each mesh cell $\xx_K$, and in the case the center of mass $\overline{\xx}_K$ is chosen, then the scheme corresponds to hMFD and superconvergence is well known~\cite{supconv1,supconv2,hmm2}. Until section \Rref{sect5}, we will consider $\xx_K=\overline{\xx}_K$. \\

\begin{dfntn} (Hybrid Mimetic Mixed gradient discretisation (HMM-GD)) \\
For hMFD scheme, we use the following GD (Definition 13.1.1~\cite{cindy}):
\begin{enumerate}
    \item Let $X_{\dd,0}= \{v=((v_K)_{K\in \mm}, (v_{\sigma})_{\sigma \in \ff}): v_K\in \mathbb{R},v_{\sigma}\in \mathbb{R},v_{\sigma}=0 \textrm{ if } \sigma \in \ff_{ext}\},$
    \item $\Pi_\dd:X_{\dd,0}\to L^2(\Omega)$ is the following piecewise constant reconstruction on the mesh:\\
$\forall v \in X_{\dd,0}, \forall K \in \mm,$
\begin{equation}
  \Pi_{\dd}v(\xx)=v_K \textrm{ on }K.\label{PiD}
\end{equation}
\item $\nabla_{\dd}:X_{\dd,0}\to L^2(\Omega)^d$ reconstructs piecewise constant gradients on the cones $(D_{K,\sigma})_{K\in\mm,\sigma \in \ff_K}$:\\
$\forall v \in X_{\dd,0}, \forall K \in \mm, \forall \sigma \in \ff,\\$
\begin{equation}
    \nabla_{\dd}v(\xx)=\nabla_{K}v+\frac{\sqrt{d}}{d_{K,\sigma}}[\mathcal{L}_K R_K(v)]_{\sigma}\ \nn_{K,\sigma} \textrm{ on } D_{K,\sigma},\label{NablaD}
\end{equation}
where:
\begin{itemize}
    \item $ \nabla_{K}v=\frac{1}{|K|}\sum_{\sigma\in \ff_K}|\sigma| v_{\sigma} \nn_{K,\sigma}$,
 \item $R_K:X_{\dd,0}\to \mathbb{R}^{\ff_K} $ is given by $ R_K(v)=(R_{K,\sigma(v)}))_{\sigma \in \ff_K}$ with $ R_{K,\sigma}(v)=v_{\sigma}-v_K-\nabla_{K}v\cdot(\overline{\xx}_{\sigma}-\xx_K)$,
 \item $\llm_K$ is an isomorphism of the space $Im(R_K)$.
 \end{itemize}
\end{enumerate}
\end{dfntn}
 As explained in the introduction of this chapter, hMFD,
HFV and MFV schemes are three different presentations of the same method. With the notations above, any HMM method for the weak form \eqref{varellip} can be written (Equation 2.25~\cite{jd}):\\
Find $u_\ttens(\mu) \in X_{\dd,0}$ such that, for all $v_\ttens \in X_{\dd,0},$
\begin{equation*}
   \mu  \underset{K\in\mm}{\sum}|K| A_K(\mu)\nabla_Ku_\ttens\cdot \nabla_K v_\ttens+\underset{K\in\mm}{\sum} R_K(v_\ttens)^T\mathbb{B}_KR_K(u_\ttens)=\underset{K\in\mm}{\sum}v_K \int_Kf(\xx) \ d\xx,
    \label{hmm}
\end{equation*}
where $A_K(\mu)$ is the $L_2$ projection of $A(\mu)$ on $K$ and $\mathbb{B}_K = ((\mathbb{B}_K)_{\sigma,\sigma'})_{\sigma,\sigma' \in \ff_K}$ is a symmetric positive definite matrix.\\
For a certain choice of isomorphism $\llm_K:\Im(R_K)\to \Im(R_K)$, the HMM scheme \eqref{hmm} is identical to GDs \eqref{variat} (see Theorem 13.7~\cite{cindy}).\\

\noindent We now introduce the super-convergence property which will be used in the proof of theorem \Rref{th11}, but first we need the following $H^2$ regularity assumption (which holds if $A$ is Lipschitz continuous and $\Omega$ is convex):\\

\noindent Let $f \in L^2(\Omega),$ the solution $u(\mu)$ to \eqref{varellip} belongs to $H^2(\Omega),$ and
\begin{equation}
    \norm{u(\mu)}_{H^2(\Omega)}+\norm{A(\mu)\nabla u(\mu)}_{H^1(\Omega)^d} \leq C \norm{f}_{L^2(\Omega)},
    \label{h2reg}
\end{equation}
with $C$ depending only on $\Omega$ and $A$.\\
%cumbersome notations=notations lourdes

We define $\pi_{\mm^h}:L^2(\Omega)\to L^2(\Omega)$ as the orthogonal projection on the piecewise constant functions on $\mm^h$ that is
\begin{equation*}
  \forall \Psi \in L^2(\Omega), \quad \forall K \in \mm^h,\quad \pi_{\mm^h}\Psi=\frac{1}{|K|}\int_K \Psi(\xx)\  d\xx \textrm{ on }K. 
\end{equation*}

\begin{thrm}[Super-convergence for hMFD schemes, Theorem 4.7~\cite{supconv2}]
Let $d\leq 3$, $f\in H^1(\Omega)$, and $u(\mu)$ be the solution of $\eqref{varellip}$ under assumption \eqref{h2reg}. Let $\mathcal{T}_h$ be a polytopal mesh, and $\mathcal{D}$ be an HMM gradient discretisation on $\mathcal{T}_h$ with the unknowns defined on $\xx_K$, and let $u_{h}(\mu)$ be the solution of the corresponding GD. Recall that $\overline{\xx}_K$ is the center of mass of $K$ and we are in the case where $\xx_K=\overline{\xx}_K$. Then, considering $u_{\mathcal{P}}(\mu)$ as the piecewise constant function on $\mathcal{M}_h$ equal to $u(\overline{\xx}_K;\mu)$ on $K\in \mm$, there exists $C>0$ not depending on $h$ such that

\begin{equation}
   \norm{\Pi_{\dd}^hu_{h}(\mu)-u_{\mathcal{P}}(\mu)}_{L^2(\Omega)}\leq C(\norm{f}_{H^1(\Omega)}+\norm{u}_{H^2(\Omega)})h^2.
    \label{superconvhmm}
\end{equation}
\end{thrm}

To recover \eqref{superconvhmm} in the case $\xx_K=\overline{\xx}_K$, we used the Lemma 7.5 of~\cite{supconv2} on the approximation of $H_2$ functions by affine functions  to obtain
\begin{equation*}
  \norm{\pi_{\mm^h}u(\mu)-u_{\mathcal{P}}(\mu)}_{L^2(\omega)}\leq Ch^2\norm{u}_{H^2(\Omega)}.
  \end{equation*}

\begin{rmrk}

 We consider here $\norm{\cdot}_\dd$ as the discrete semi norm of $H^1$ so as not to make notations too cumbersome. The usual discrete semi-norm for $H^1$ is defined by
\begin{equation}
    \forall v \in \mathcal{T},\ |v|_{\ttens,2}^2= \underset{K \in \mm}{\sum} \underset{\sigma \in \ff_K}{\sum} |\sigma| d_{K,\sigma}\left\rvert\frac{v_{\sigma}-v_K}{d_{K,\sigma}}\right\rvert^2. \label{discretenorme}
\end{equation}
Under some conditions on the regularity of the mesh, this norm and $\norm{\nabla_\dd \cdot}_{L^2(\Omega)^d}$ are equivalent (Lemma 13.11~\cite{cindy}).
\end{rmrk}

In the next section, we recall the offline and the online parts of the two-grids algorithm. 

\section{The Non Intrusive Reduced Basis method (NIRB)}
\label{sect3}
This section recalls the main steps of the two-grids method algorithm~\cite{madaychakir, NIRB1}.\\ 

\noindent Let $u_h(\mu)$ refer to the hMFD solution on a fine polytopal mesh $\mathcal{T}_h$, with cells $\mathcal{M}_h$ and respectively $u_H(\mu)$ the one on a coarse mesh $\mathcal{T}_H$, with the cells $\mathcal{M}_H$.\\

\noindent We briefly recall the NIRB method. Points 1 and 2 are in the offline part, and the others are done online. 
\begin{enumerate}
    \item Several snapshots $\{\uu_{h}(\mu_i)\}_{i \in \{1,\dots N\}}$ are computed with the hMFD scheme $\eqref{variat}$, where $\mu_i\in \mathcal{G}\quad \forall i=1,\cdots,N$. The space generated by the snapshots is named $X_h^N = Span\{ u_h(\mu_1 ), \dots ,  u_h(\mu_N)\}$.
    \item We generate the basis functions $(\Phi_i^h)_{i=1,\cdots,N}$ with the following steps:
    \begin{itemize}
    \item A Gram-Schmidt procedure is used, which involves $L^2$ orthonormalization of the reconstruction functions. 
    \item This procedure is also completed by the following eigenvalue problem:
    \begin{numcases}
      \strut \textrm{Find } \Phi^h \in X_h^N, \textrm{ and } \lambda \in \mathbb{R} \textrm{ such that: }    \nonumber\\
        \forall \vv \in X_h^N, \int_{\Omega} \nabla_\dd^h \Phi^h \cdot \nabla_\dd^h \vv \ d\xx= \lambda \int_{\Omega} \Pi_\dd^h \Phi^h \cdot \Pi_\dd^h \vv \ d\xx, \label{orthoTPFA} 
    \end{numcases} 
    
    where $\nabla_\dd^h$ and $\Pi_\dd^h$ are respectively the discrete gradient and the discrete reconstruction operators as in the definition of the HMM GD (\eqref{PiD}, \eqref{NablaD}). We get an increasing sequence of eigenvalues $\lambda_i$, and orthogonal eigenfunctions $(\Pi_\dd^h\Phi_i^h)_{i=1,\cdots,N}$, orthonormalized in $L^2(\Omega)$ and orthogonalized in $H^1(\Omega)$, such that ($\Phi_i^h)_{i=1,\cdots,N}$ defines a new basis of the space $X_h^N$.
    \end{itemize}
    \item We solve the hMFD problem \eqref{variat} on the coarse mesh $\mathcal{T}_H$ for a new parameter $\mu \in \mathcal{G}$. Let us denote by $u_H(\mu)$ the solution.
    \item We then introduce $\alpha_i^H(\mu)=\int_{\Omega}\Pi_\dd^H \uu_H(\mu) \cdot \Pi_\dd^h \Phi_i^h\ d\xx$. The approximation used in the two-grids method is $\uu_{Hh}^N(\mu)=\overset{N}{\underset{i=1}{\sum}}\alpha_i^H(\mu) \Pi_\dd^h \Phi_i^h.$
    \end{enumerate}
    In the next section, we detail how to obtain the classical finite elements estimate in $\mathcal{O}(h)$ on the NIRB algorithm, when the snapshots are computed with the hMFD GD using a polytopal mesh.
    \section{NIRB error estimate}
    \label{sect4}
   In this section, we consider $\xx_K=\overline{\xx}_K$ which is the case with the hMFD scheme. Some other cases will be detailed in section \Rref{sect5}. \\
We now continue with the proof of theorem \Rref{th11}. 

\begin{proof}
In this proof, we will denote $A \lesssim B$ for $A \leq  CB$ with $C$ not depending on $h$ or $H$.

We use the triangle inequality on $\norm{\uu(\mu)-\uu_{Hh}^N(\mu)}_{\dd}$ to get
\begin{align}
\norm{\uu(\mu)-\uu_{Hh}^N(\mu)}_{\dd}&\leq \norm{\uu(\mu)-\Pi_\dd^h \uu_{h}(\mu)}_{\dd}+\norm{\Pi_\dd^h\uu_h(\mu)-\uu_{hh}^N(\mu)}_{\dd}+ \norm{\uu_{hh}^N(\mu)-\uu_{Hh}^N(\mu)}_{\dd} \nonumber \\
&=:T_1+T_2+T_3, 
\label{triangleinequality}
\end{align}
where $\uu_{hh}^N(\mu)=\overset{N}{\underset{i=1}{\sum}}\alpha_i^h(\mu) \Pi_\dd^h \Phi_i^h,$ and $\alpha_i^h(\mu)=\int_{\Omega}\Pi_\dd^h\uu_h(\mu) \cdot \Pi_\dd^h \Phi_i^h\ d\xx$.\\
\begin{itemize}
    \item The first term $T_1$ can be estimated using a classical result for finite volume schemes (Consequence of Proposition 13.16~\cite{cindy}) such that: 
\begin{equation}
    \norm{\uu(\mu)-\Pi_\dd^h \uu_{h}(\mu)}_{\dd} \lesssim  h\norm{u}_{H^2(\Omega)}.
    \label{VFestim}
\end{equation}
%comparer N à epsilon?
\item The best achievable error in the uniform sense of a fine solution projected into $X_h^N$ relies on the notion of Kolmogorov n-width (Theorem 20.1~\cite{nummod}). 
If $\mathcal{K}$ is a compact set in a Banach space $V$, the Kolmogorov n-width of $\mathcal{K}$ is \begin{equation}
    d_n(\mathcal{K})=\underset{\textrm{ dim} (V_n)\leq n}{\inf}\quad \underset{v\in \mathcal{K}}{\sup}\quad \underset{w\in V_n}{\min} \norm{v-w}_{V}.
\end{equation} 
Here we suppose the set of all the reconstructions of the solutions $\mathcal{S}=\{\Pi_\dd^h \uu_h(\mu),\mu \in \mathcal{G}\}$ has a low complexity which means for an accuracy $\varepsilon=\varepsilon(N)$ related to the Kolmogorov n-width of the manifold $\mathcal{S}$, there exists a set of parameters $\{\mu_1,\dots,\mu_N\} \in \mathcal{G}$, such that~\cite{cohen,madaychakir,NIRB1,greedy}
\begin{equation}
  T_2=\norm{\Pi_\dd^h \uu_h(\mu)-\overset{N}{\underset{i=1}{\sum}}\alpha_i^h(\mu)\Pi_\dd^h \Phi_i^h}_{\dd} \leq \varepsilon(N).\label{kolmotpfa}
  \end{equation}

\item Consider the term $T_3$ now. We will need the following proposition where the property of super-convergence for the hMFD scheme \eqref{superconvhmm} is used.\\
%A classical result of FE on projection operators \cite{aubin} is:\\
%\begin{equation}
%    \norm{\Phi-\Pi_k^h \Phi}_{H^{m}(\Omega)}\leq Ch^{k+1-m}|\Phi|_{H^{k+1}(\Omega)},
%\end{equation}
%with $m\leq k+1$, for all $\Phi \in H^{k+1}(\Omega)$, where $\Pi_k^h \Phi(\xx)=\Phi(\xx) \textrm{ on } K\in \mathcal{M}_h, \textrm{ with } \Phi_{|K} \in \mathbb{P}_k, \forall K \in \mathcal{M}_h$, and $k> 0$.\\
%In particular, we get for the previous projection operators $\Pi_0^h$ and $\Pi_1^h$ the followings inequalities with $L^2$ norm:\\
%From equation 12, and Cours M2? ou equation 11 direct? 
%\begin{subnumcases}
%    \strut \norm{\Phi-\Pi_0^h \Phi}_{L^2}\leq Ch |\Phi|_{H^2},\label{pi0eq}\\
%   \label{pi1eq} \norm{\Phi-L_{\Phi}}_{L^2}\leq Ch^2|\Phi|_{H^2(\Omega)}.
%\end{subnumcases}

\begin{prpstn}
Let $\uu_H(\mu)$ be the solution of the hMFD on a polytopal mesh $\mathcal{T}_H$ with the unknowns on $\xx_K=\overline{\xx}_K$. Denote by $\uu(\mu)$ the exact solution of equation  $\eqref{varellip}$, and let $(\Phi_i^h)_{i=1,\cdots,N}$ be the basis functions of the NIRB algorithm, then there exists a constant $C=C(N)>0$ not depending on $H$ or $h$,and depending on $N$ such that
\begin{equation}
    \left\rvert\int_{\Omega}(\uu(\mu)-\Pi_\dd^H \uu_H(\mu))\cdot \Pi_\dd^h \Phi_i^h \ d\xx \right\rvert\lesssim ((\norm{\Phi_i}_{L^{\infty}(\Omega)}+C(N))\norm{u}_{H^2(\Omega)}+\norm{f}_{H^1(\Omega)})H^2. \label{EstimH2prop}
\end{equation}
\end{prpstn}
\begin{proof}
Since $\mm_H$ is a partition of $\Omega$,  
\begin{align}
    \int_{\Omega}\Pi_\dd^H \uu_H(\mu) \cdot \Pi_\dd^h \Phi_i^h \ d\xx &=\underset{K\in \mathcal{M}_H}{\sum} \int_K \Pi_\dd^H u_H(\mu)  \cdot \Pi_\dd^h \Phi_{i}^h \ d\xx. \label{avecH}
\end{align}
To begin with, let $\Pi_0^H:\mathcal{C}(\Omega)\to L^\infty(\Omega)$ be the piecewise constant projection operator on $\mathcal{M}_H$ such that:
\begin{equation}
     \Pi_0^H \Phi(\xx)=\Psi(\xx_K),\quad \textrm{ on } K ,\quad \forall K \in \mathcal{M}_H,\quad  \forall \Psi \in \mathcal{C}(\Omega).\label{pi0}\\
\end{equation}
We use the triangle inequality on the left part of the inequality \eqref{EstimH2prop} and therefore,  
\begin{align}
\left\lvert \int_{\Omega}(\uu(\mu)-\Pi_\dd^H \uu_H(\mu))\cdot \Pi_\dd^h \Phi_i^h \ d\xx \right\rvert &\leq \left\rvert \int_{\Omega} (u(\mu)-\Pi_0^H u(\mu))\cdot \Pi_\dd^h \Phi_i^h \ d\xx \right\rvert + \left\rvert \int_{\Omega} (\Pi_0^H u(\mu)-\Pi_\dd^H u_H(\mu))\cdot\Pi_\dd^h \Phi_i^h \ d\xx \right\rvert,\nonumber \\
&=:T_{3,1} + T_{3,2}. \label{decompose1}
\end{align}

\begin{itemize}

    \item We first consider the term $T_{3,1}$, but beforehand the estimate of $T_{3,1}$ requires the use of a further operator which we now introduce. 
  Each cell $K\in \mm_H$ is star-shaped with respect to a ball $B_K$ centered in $\xx_K$ of radius $\rho=\underset{\sigma \in \ff_K}{\textrm{min }} d_{K,\sigma}$ (Lemma B.1~\cite{cindy}). We then use an averaged Taylor polynomial as in~\cite{FE} but simplified. Let us consider the following polynomial of $u(\mu)$ averaged over $B_K$:\\
\begin{equation}
    Q_K u(\xx;\mu)=\frac{1}{|B_K|}\int_{B_K} [u(\mathbf{y};\mu)+D^1 u(\mathbf{y};\mu)(\xx-\mathbf{y})] \ d\mathbf{y}.\\
    \label{defQ}
\end{equation}
%\supset B(\xx_K, \theta^{-1}_\ttens H_K)$
%$\theta \geq \frac{\textrm{diam }K}{\textrm{diam } B_K}$
This polynomial is of degree less or equal to $1$ in $\xx$.\\
Let us introduce $\Pi_1^H:H^1(\Omega) \cap \mathcal{C}(\Omega)\to \mathbb{R}$, the piecewise affine projection operator on $\mathcal{M}_H$ such that:
\begin{equation}
     \Pi_1^H \Psi=Q_K \Psi(\xx), \quad \textrm{ on } K, \quad \forall K \in \mathcal{M}_H, \quad \forall \Psi \in H^1(\Omega)\cap \mathcal{C}.\label{pi01}\\
\end{equation}
With the triangle inequality, we obtain \begin{align}
    T_{3,1} &\leq \left\rvert\int_{\Omega} (u(\mu)-\Pi_1^Hu(\mu))\cdot \Pi_\dd^h \Phi_i^h \ d\xx \right\rvert + \left\rvert\int_{\Omega} (\Pi_1^Hu(\mu)-\Pi_0^H u(\mu))\cdot\Pi_\dd^h \Phi_i^h \ d\xx \right\rvert,\nonumber \\
    &=: T_{3,1,1} + T_{3,1,2}. \label{decompose2}
\end{align}
 
  \begin{itemize}
       \item 
    Using the Cauchy-Schwarz inequality,
\begin{align}
    T_{3,1,1} & \leq\int_{\Omega} \left\rvert (u(\mu)-\Pi_1^Hu(\mu))\cdot \Pi_\dd^h \Phi_i^h \right\rvert\ d\xx , \nonumber \\ 
    &\leq \norm{\uu(\mu)-\Pi_1^Hu(\mu)}_{L^2(\Omega)}\norm{\Pi_{\dd}^h \Phi_i^h}_{L^2(\Omega)},  \nonumber \\ 
    \label{331e} &\leq \norm{\uu(\mu)-\Pi_1^Hu(\mu)}_{L^2(\Omega)}, \textrm{ since $\Pi_\dd^h\Phi_i^h\quad  \forall i=1,\cdots,N$ are normalized in } L^2.
    \end{align}
    
Let $K\in\mm_H$. As in Proposition 4.3.2~\cite{FE},
\begin{equation}
\underset{\xx\in \overline{K}}{\sup\ } |u(\xx;\mu)-Q_K u(\xx;\mu)|\lesssim H_K^{2-\frac{d}{2}} |u(\mu)|_{H^2(K)}.
\label{supestim}
\end{equation}
Since $K\subset B(\xx,H)$ for all $\xx \in K$, 
\begin{align}
    |K|&\leq |B(\xx_K,H)|=|B(0,1)|H_K^d. \label{K}
\end{align}
Thus, with the inequalities \eqref{K} and \eqref{supestim}, we get  \begin{equation}
    \underset{\xx\in \overline{K}}{\sup\ } |u(\xx;\mu)-Q_K u(\xx;\mu)|\lesssim H_K^{2}|K|^{-\frac{1}{2}} |u(\mu)|_{H^2(K)},
\end{equation}
taking the square and integrating over $K$, we obtain
\begin{equation}
    \int_K |u(\mu)-\Pi_1^H u(\mu)|^2 \ d\xx \lesssim H_K^{4} |u(\mu)|_{H^2(K)}^2,
\end{equation}
and summing over $K$ yields 
\begin{equation}
    \norm{u(\mu)-\Pi_1^H u(\mu)}_{L^2(\Omega)}\lesssim H^2 |u(\mu)|_{H^2(\Omega)}.
    \label{bramblehilbert}
\end{equation}
The inequality \eqref{bramblehilbert}, combined with \eqref{331e}, entails that
\begin{equation}
    T_{3,1,1}  \lesssim H^2 |u(\mu)|_{H^2(\Omega)}.
    \label{311estim}
\end{equation}

\item The term $T_{3,1,2}$ can be estimated using a continuous reconstruction of $\Phi_i^h$, denoted by $\Phi_i$ .\\
With the triangle inequality,
\begin{align}
  \left\rvert \int_{\Omega} (\Pi_1^H \uu(\mu)-\Pi_0^H u (\mu))\cdot \Pi_{\dd}^h \Phi_i^h \ d\xx\right\rvert&\leq \left\rvert \int_{\Omega}(\Pi_1^Hu(\mu)-\Pi_0^H u(\mu))( \Pi_{\dd}^h\Phi_i^h-\Pi_0^H\Phi_i) \ d\xx \right\rvert \nonumber \\
  &+ \left\rvert \int_{\Omega} (\Pi_1^Hu(\mu)-\Pi_0^H u(\mu))\cdot\Pi_0^H\Phi_i) \ d\xx \right\rvert. \label{inegalitétriangulairesurT312}
  \end{align}

Since $\xx_K$ is the center of mass, $\int_{K} \xx \ d\xx =|K|\xx_K$. Therefore,
\begin{equation}
    \int_{K}Q_K u(\xx;\mu)\ d\xx = |K| Q_K u (\xx_K; \mu).
    \label{centerofmass}
\end{equation}
From the inequality \eqref{supestim}, 
\begin{align}
    |Q_K u(\xx_K;\mu)-u(\xx_K;\mu)|\lesssim H_K^{2-\frac{d}{2}} |u(\mu)|_{H^2(K)}.
    \label{estimsurxk}
\end{align}

Thus, since $\Pi_0^H\Phi_i$ is constant on each cell $ K \in \mm_H$, and $|K|\lesssim H_K^d$ \eqref{K},
\begin{align}
    \left\rvert \int_{\Omega} (\Pi_1^H u(\mu)-\Pi_0^Hu(\mu)) \cdot \Pi_0^H\Phi_i \ d\xx \right\rvert &= \left\rvert \underset{K\in \mathcal{M}_H}{\sum}  \int_K (Q_K u(\xx;\mu)-u(\xx_K;\mu)) \cdot \Pi_0^H \Phi_i \ d\xx  \right\rvert, \nonumber \\
    &\leq \underset{K\in \mathcal{M}_H}{\sum}  \left\rvert  \Phi_i(x_K) \int_{K} Q_Ku(\xx;\mu) -u(\xx_K;\mu)  \ d\xx  \right\rvert, \nonumber \\
    &\leq \underset{K\in \mathcal{M}_H}{\sum} |K| \left\rvert  \Phi_i(x_K) (Q_Ku(\xx_K;\mu) -u(\xx_K;\mu))  \right\rvert,   \textrm{ from \eqref{centerofmass}},\nonumber \\
    &\leq  \norm{\Phi_i}_{L^{\infty}(\Omega)}\underset{K\in \mathcal{M}_H}{\sum}  |K|\left\rvert Q_Ku(\xx_K;\mu)  - u(\xx_K;\mu) \right\rvert, \nonumber  \\
&\lesssim \norm{\Phi_i}_{L^{\infty}(\Omega)} \underset{K\in \mathcal{M}_H}{\sum}  |K|H_K^{2-\frac{d}{2}}|u(\mu)|_{H^2(K)} \textrm{ from \eqref{estimsurxk}}, \nonumber \\
&\lesssim \norm{\Phi_i}_{L^{\infty}(\Omega)} \underset{K\in \mathcal{M}_H}{\sum}  H_K^{2+\frac{d}{2}}|u(\mu)|_{H^2(K)}. \label{test}
\end{align}
Since Card$(\mm_H)\simeq H^{-d} $, using the Cauchy-Schwarz inequality, the inequality \eqref{test} becomes 

\begin{align}
   \left\rvert \int_{\Omega} (\Pi_1^H u(\mu)-\Pi_0^Hu(\mu)) \cdot \Pi_0^H\Phi_i \ d\xx \right\rvert &\lesssim \norm{\Phi_i}_{L^{\infty}}H^{2} (\underset{K\in \mathcal{M}_H}{\sum} |u(\mu)|_{H^2(K)}^2)^{\frac{1}{2}}, \nonumber \\
    &=\norm{\Phi_i}_{L^{\infty}}|u(\mu)|_{H^2(\Omega)}H^2 \label{inter},
\end{align}
 
which implies that there exists a constant $\widetilde{C_1}>0$ not depending on $h$ or $H$ such that  \eqref{inegalitétriangulairesurT312} becomes 
%\begin{equation}
  %    \left\rvert \int_{\Omega} (\Pi_1^H \uu(\mu)-\Pi_0^H u (\mu))\cdot \Pi_{\dd}^h \Phi_i^h \ d\xx\right\rvert
   %   \leq \left\rvert \int_{\Omega}(\Pi_1^H u(\mu)-\Pi_0^H u(\mu))( \Pi_{\dd}^h\Phi_i^h-\Pi_0^H\Phi_i) \ d\xx\right\rvert + CH^2.
%\end{equation}
  \begin{equation}
      T_{3,1,2} \leq  \int_{\Omega} \left\rvert(\Pi_1^H u(\mu)-\Pi_0^H u(\mu))( \Pi_{\dd}^h\Phi_i^h-\Pi_0^H\Phi_i) \right\rvert d\xx  + \widetilde{C_1}\norm{\Phi_i}_{L^{\infty}}|u(\mu)|_{H^2(\Omega)}H^2.
      \label{derniereestimsurT312}
  \end{equation}
  
From the Cauchy-Schwarz inequality and the inequality \eqref{derniereestimsurT312},
\begin{align}
        T_{3,1,2}
        &\leq \norm{\Pi_1^H\uu(\mu)-\Pi_0^H u(\mu)}_{L^2(\Omega)}\norm{\Pi_{\dd}^h \Phi_i^h-\Pi_0^H\Phi_i}_{L^2(\Omega)} + \widetilde{C_1}\norm{\Phi_i}_{L^{\infty}}|u(\mu)|_{H^2(\Omega)}H^2.        \label{enoH2}
\end{align}
From Bramble-Hilbert's Lemma (see~\cite{FE}), we deduce that
\begin{equation}
    \norm{u(\mu)-\Pi_0^H u(\mu)}_{L^2(\Omega)}\lesssim H \norm{u(\mu)}_{H^2(\Omega)}.\label{pi0eq}
\end{equation}
For the first term in the right-hand side of \eqref{enoH2}, from \eqref{bramblehilbert}-\eqref{pi0eq} and the triangle inequality,  
\begin{align}
    \norm{\Pi_1^H \uu(\mu)- \Pi_0^H u (\mu)}_{L^2(\Omega)} &\leq \norm{\Pi_1^H \uu(\mu)- u (\mu)}_{L^2(\Omega)} + \norm{u(\mu)- \Pi_0^H u (\mu)}_{L^2(\Omega)}, \nonumber \\
    &\lesssim H\norm{u}_{H^2(\Omega)}, \textrm{ neglecting the estimate in $H^2$}, \label{1erterme}
\end{align} 
and the inequality \eqref{pi0eq} and the classical finite volume estimate as for \eqref{VFestim} ($\Pi_{\dd}^h \phi_i^h$ being a linear combination of the family $(\Pi_{\dd}^h\uu_j^h)_{j=1}^N,\ \forall i=1,\cdots,N$) implies that there exists $\widetilde{C_2}=\widetilde{C_2}(N)>0$ not depending of $H$ or $h$ but depending on $N$ such that
    \begin{align}
      \norm{\Pi_{\dd}^h \Phi_i^h-\Pi_0^H \Phi_i}_{L^2(\Omega)} & \leq \norm{\Pi_{\dd}^h \Phi_i^h-\Phi_i}_{L^2(\Omega)}+\norm{\Phi_i-\Pi^H_0 \Phi_i}_{L^2(\Omega)}, \nonumber \\
      &\leq \widetilde{C_2}(N) H, \textrm{ neglecting the estimate in $h$}. \label{2emeterme}
\end{align}
From \eqref{1erterme}-\eqref{2emeterme}, we deduce that each $L^2$ term is in $\mathcal{O}(H)$ in the product of the right-hand side of \eqref{enoH2}. Hence the equation \eqref{inegalitétriangulairesurT312} yields to 
\begin{equation}
\label{premier2terme}
     T_{3,1,2}=\left\rvert \int_{\Omega} (\Pi_1^H \uu(\mu)-\Pi_0^H u(\mu))\cdot \Pi_{\dd}^h \Phi_i^h \ d\xx \right\rvert \lesssim (\widetilde{C_1}\norm{\Phi_i}_{L^{\infty}(\Omega)}+\widetilde{C_2}(N)) \norm{u}_{H^2(\Omega)}H^2.
\end{equation}
   \end{itemize}

\item We now proceed with the estimate on $T_{3,2}:$\\
With the super-convergence property on the hMFD scheme \eqref{superconvhmm}, and with the normalization of $\Pi_{\dd}^h \Phi_i^h $ in $L^2(\Omega)$
\begin{align}
    \left\rvert \int_{\Omega} (\Pi_{\dd}^H \uu_H(\mu)-\Pi_0^H \uu(\mu)) \cdot \Pi_{\dd}^h \Phi_i^h \ d\xx \right\rvert & \leq  \int_{\Omega} \left \rvert (\Pi_{\dd}^H \uu_H(\mu)-\Pi_0^H \uu(\mu)) \cdot \Pi_{\dd}^h \Phi_i^h\right\rvert \ d\xx , \nonumber \\
    &\leq \norm{\Pi_{\dd}^H \uu_H(\mu)-\Pi_0^H \uu(\mu)}_{L^2(\Omega)}\norm{\Pi_{\dd}^h \Phi_i^h}_{L^2(\Omega)}, \nonumber \\
    &\lesssim (\norm{f}_{H^1(\Omega)} +\norm{u}_{H^2(\Omega)})H^2. \label{term2bis}
\end{align}
\end{itemize}

%From \eqref{orthoTPFA}, we get that 
%\begin{align}
%\label{lambdamaxortho}
%    |\Phi_i^h|_{\dd}^2&\leq C \int_{\Omega} |\nabla_\dd \Phi_i^h|^2,\\
%    &=\lambda_i C \norm{\Pi_\dd \Phi_i^h}_{L^2(\Omega)}^2\\
 %   &\leq \lambda_i C \leq C \underset{i=1,\cdots,N}{\max}(\lambda_i)=C \lambda_N,
%\end{align} such that, we obtain
%As in \cite{lap}, we define the discrete laplacian $\Delta_K \uu$, constant on each cell, with the equation:  \begin{equation}
%    -\sum_{K\in \mathcal{M}} |K| \vv_K \nabla_K \uu = \int_{\Omega} \nabla_{\dd} \uu \cdot \nabla_{\dd} \vv \forall \uu,\vv \in 
%\end{equation}
Combining the estimates \eqref{311estim}-\eqref{premier2terme}-\eqref{term2bis} with the inequalities \eqref{decompose1}-\eqref{decompose2}, this results in the inequality \eqref{EstimH2prop}.
\end{proof}

%\vspace{1cm}
\noindent We now consider the third term $T_3=\norm{\uu_{hh}^N(\mu)-\uu_{Hh}^N(\mu)}_{\dd}$.
\begin{align}
    T_3&= \norm{\overset{N}{\underset{i=1}{\sum}}\alpha_i^h(\mu) \Pi_{\dd}^h \Phi_i^h-\overset{N}{\underset{i=1}{\sum}}\alpha_i^H(\mu) \Pi_{\dd}^h \Phi_i^h}_{\dd}, \nonumber \\ 
    &\leq \overset{N}{\underset{i=1}{\sum}} \left\rvert \alpha_i^h(\mu)-\alpha_i^H(\mu)\right\rvert \norm{\Pi_{\dd}^h \Phi_i^h}_{\dd}, \nonumber \\
     &= \overset{N}{\underset{i=1}{\sum}} \left\rvert(\Pi_{\dd}^h \uu_h(\mu)-\Pi_{\dd}^H \uu_H(\mu),\Pi_{\dd}^h \Phi_i^h)_{L^2}\right\rvert \norm{\Pi_{\dd}^h \Phi_i^h}_{\dd}. \label{on sensert}
\end{align}
From \eqref{orthoTPFA}, we get that 
\begin{align}
   \norm{\Pi_\dd^h\Phi_i^h}_{\dd}^2\ &=\ \int_{\Omega} |\nabla_\dd \Phi_i^h|^2 \ d\xx \ = \ \lambda_i \norm{\Pi_\dd \Phi_i^h}_{L^2(\Omega)}^2 \leq  \underset{i=1,\cdots,N}{\max}(\lambda_i)= \lambda_N.\label{lambdamaxortho}
\end{align}
Therefore we obtain from \eqref{on sensert} and \eqref{lambdamaxortho},
\begin{equation}
    T_3\leq \sqrt{\lambda_N} \overset{N}{\underset{i=1}{\sum}} \left\rvert (\Pi_{\dd}^h \uu_h(\mu)-\Pi_{\dd}^H \uu_H(\mu),\Pi_{\dd}^h \Phi_i^h)_{L^2}\right\rvert. \label{on sensert bis}\\
\end{equation}
Using the triangle inequality in the right-hand side of \eqref{on sensert bis}, 
\begin{equation}
    T_3\leq \sqrt{\lambda_N} \overset{N}{\underset{i=1}{\sum}} \left\rvert(\Pi_{\dd}^h \uu_h(\mu)-u(\mu),\Pi_\dd^h \Phi_i^h)\right\rvert+\left\rvert(u(\mu)-\Pi_{\dd}^H \uu_H(\mu),\Pi_{\dd}^h \Phi_i^h)\right\rvert. \label{on sensert bisbis}\\
\end{equation}
From Proposition 1, with the estimate \eqref{EstimH2prop} applied to $\mm_h$ and $\mm_H$, neglecting the estimate in $\mathcal{O}(h^2)$
\begin{equation}
    T_3 \lesssim \sqrt{\lambda_N}N((\norm{\Phi_i}_{L^{\infty}(\Omega)}+C(N))\norm{u}_{H^2(\Omega)}+\norm{f}_{H^1(\Omega)}) H^2.
    \label{T3}
\end{equation}
\end{itemize}
The conclusion follows combining the estimates on $T_1, T_2$ and $T_3$ (estimates \eqref{VFestim},\eqref{kolmotpfa} and \eqref{T3}).
\begin{align}
\norm{u(\mu)-u_{Hh}^N(\mu)}_\dd &= \norm{\uu(\mu)-\sum_{i=1}^N  \alpha_i^H(\mu) \Pi_{\dd}^h \Phi^h_i}_{\dd}, \nonumber \\
& \leq \varepsilon(N) + C_1 h + C_2(N) H^2 \sim \mathcal{O}(h) \textrm{ if }h \sim H^2.
\end{align}
\end{proof}
\section{Results on other FV schemes}
\label{sect5}
In this section, we consider the case where $\xx_K$ is not the center of mass, as it is the case for some FV schemes. Therefore the left hand side of the inequality \eqref{test} cannot be estimated using equation \eqref{centerofmass}.
The unknowns $\xx_K$ are not necessarily the centers of mass of the cells neither with HMM methods nor with the Two-Point Flux Approximation (TPFA) scheme~\cite{boyer,tpfascheme}. Under the following superadmissibility condition \begin{equation}
    \forall K\in \mm_H, \ \sigma \in \ff_K: \ \nn_{K,\sigma}=\frac{\overline{\xx}_{\sigma}-\xx_K}{d_{K,\sigma}},
\end{equation} the TPFA scheme is a member of the the HMM family schemes (~\cite{cindy} section 13.3 ,\ ~\cite{jd} section 5.3) with the choice $\mathcal{L}_K=Id$. This leads to take $\xx_K$ as the circumcenters of the cells with 2D  triangular meshes. Theorem 1.1 holds in 2D on uniform rectangles with TPFA since the superadmissibility condition is satisfied in this case where $\xx_K$ is the centre of mass of the cells. The TPFA scheme is rather simple to implement, and therefore we will present in the last section  numerical results with a TPFA solver.
 We will use the definition of a local grouping of the cells as in~\cite{supconv2} (Definition 5.1).
We will extend the Theorem 1.1 in the case where such groupings of cells exist. 

\begin{dfntn} (Local grouping of the cells). Let $\mathcal{T}_H$ be a polytopal mesh of $\Omega$. A local grouping of the cells of $\mathcal{T}_H$ is a partition $\mathfrak{G}$ of $\mm_H$, such that for each $G\in \mathfrak{G}$, letting $U_G:= \underset{K \in G}{\cup} K$, there exists a ball $B_G\subset U_G$ such that $U_G$ is star-shaped with respect to $B_G$. This implies that for all $\xx \in U_G$ and all $\mathbf{y} \in B_G$, the line segment $[\xx,\mathbf{y}]$ is included in $U_G$. We then deﬁne the regularity factor of $\mathfrak{G}$ \begin{equation}
    \mu_G:= \quad \underset{G\in \mathfrak{G}}{\max }\quad \textrm{Card}(G)\quad +\quad \underset{G\in \mathfrak{G}}{\max}\quad \underset{K\in G}{\max}\quad \frac{H_K}{\textrm{diam}(B_G)},
\end{equation}

and, with $e_K=\overline{\xx}_K-\xx_K$, and
\begin{equation}
  e_{G}:=\frac{1}{|U_G|}\ \underset{K\in G}{\sum}\ |K|\ \eee_K ,\quad \forall G \in \mathfrak{G},
\end{equation}

\begin{equation}
    e_{\mathfrak{G}}:= \underset{G \in \mathfrak{G}}{\max} \quad \left\rvert \eee_G \right\rvert.
\end{equation}
\end{dfntn}
\noindent
Note that we are interested in situations where $|\mathbf{e}_{G}|=\left\rvert\frac{1}{|U_G|}\sum_{K\in G} |K|\ \mathbf{e}_K\right\rvert$ is much smaller than $|\eee_K|\quad \forall K\in G$.
The aim of this section is to estimate the left hand side of the inequality $\eqref{test}$ in $\mathcal{O}(H^2)$ using a local grouping of the cells. The rest of the proof remains unchanged. 

We will need the following Theorem of super-convergence for HMM schemes with local grouping (Theorem 5.4~\cite{supconv2}).

\begin{thrm}[Super-convergence for HMM schemes with local grouping (Theorem 5.4 ~\cite{supconv2})]
Let $f\in H^1(\Omega)$, and $u(\mu)$ be the solution of $\eqref{varellip}$ under assumption \eqref{h2reg}. Let $\mathcal{T}_h$ be a polytopal mesh, and $\mathcal{D}$ be an HMM gradient discretisation on $\mathcal{T}_h$ and $e_{\mathfrak{G}}$ be a local grouping, and let $u_{h}(\mu)$ be the solution of the corresponding GD. Then, considering $u_{\mathcal{P}}(\mu)$ as the piecewise constant function on $\mathcal{M}_h$ equal to $u(\xx_K;\mu)$ on $K\in \mm$, there exists $C$ not depending on $H$ or $h$ such that 
\begin{equation}
    \norm{\Pi_{\dd}^hu_{h}(\mu)-u_{\mathcal{P}}(\mu)}_{L^2(\Omega)}\leq C \norm{f}_{H^1(\Omega)}(h^2+e_{\mathfrak{G}}) .
    \label{superconvhmm2}
\end{equation}
\end{thrm}
\begin{thrm}[NIRB error estimate with local grouping] Let $u_{hH}^N(\mu)$ be the reduced solution projected on the fine mesh and generated with the hMFD solver with the unknowns defined on $\xx_k$ such that $e_{\mathfrak{G}}$ is in $\mathcal{O}(H^2)$ on the coarse mesh, and $u(\mu)$ be the exact solution of \eqref{varellip} under assumption \eqref{h2reg}, then the following estimate holds \\
\begin{equation}
    \norm{u(\mu) - u_{hH}^N(\mu)}_{\dd}\leq \varepsilon(N) +C_1 h + C_2(N) H^2, 
   \label{estimationNIRBBB}
\end{equation}
where $C_1$ and $C_2$ are constants independent of $h$ and $H$,$C_2$ depends on $N$, the number of functions in the basis, and $\norm{\cdot}_{\dd}$ is the discrete norm introduced in section \Rref{sect2}, and $\varepsilon$ depends of the Kolmogorov n-width. If $H$ is such as $H^2\sim h$, and $\varepsilon(N)$ small enough, it results in an error estimate in $\mathcal{O}(h)$.

\end{thrm}
\begin{proof}
In this proof, we will still denote $A\lesssim B $ for $A \leq C B$ with $C$ not depending on $h$ or $H$. The reconstruction $\Phi_i$ of $\Phi_i^h$ must belong to $W^{1,\infty}.$
As in the previous section, with the equation \eqref{centerofmass},
\begin{align}
    \left\rvert \int_{\Omega} (\Pi_1^H u(\mu)-\Pi_0^Hu(\mu)) \cdot \Pi_0^H\Phi_i \ d\xx \right\rvert &= \left\rvert \underset{K\in \mathcal{M}_H}{\sum}  \int_K (Q_K u(\xx;\mu)-u(\xx_K;\mu)) \cdot \Pi_0^H \Phi_i \ d\xx  \right\rvert, \nonumber \\
    &=  \left\rvert \underset{K\in \mathcal{M}_H}{\sum} \Phi_i(\xx_K) |K| \Big[Q_Ku(\overline{\xx}_K;\mu) -u(\xx_K;\mu)\Big] \right\rvert, \nonumber \\
    &\leq \left\rvert \underset{K\in \mathcal{M}_H}{\sum} \Phi_i(\xx_K) |K| \Big[Q_Ku(\overline{\xx}_K;\mu)  - Q_Ku(\xx_K;\mu)\Big] \right\rvert \nonumber \\&+  \norm{\Phi_i}_{L^{\infty}(\Omega)} \underset{K\in \mathcal{M}_H}{\sum}  \left\rvert Q_Ku(\xx_K;\mu) -u(\xx_K;\mu) \right\rvert \textrm{ from the triangle inequality}. 
    \label{test2}
\end{align}

As in the previous section \eqref{inter},
\begin{equation}
    \norm{\Phi_i}_{L^{\infty}(\Omega)}\underset{K\in \mathcal{M}_H}{\sum} |K| \left\rvert Q_Ku(\xx_K;\mu)-u(\xx_K;\mu)\right\rvert \lesssim \norm{\Phi_i}_{L^{\infty}(\Omega)}\norm{u}_{H^2(\Omega)}H^2 \label{inter22}.
\end{equation}
Thus, the inequality \eqref{test2} yields
\begin{equation}
  \left\rvert \int_{\Omega} (\Pi_1^H u(\mu)-\Pi_0^Hu(\mu)) \cdot \Pi_0^H\Phi_i \ d\xx \right\rvert \lesssim\left\rvert \underset{K\in \mathcal{M}_H}{\sum} \Phi_i(\xx_K) |K| \Big[Q_Ku(\overline{\xx}_K;\mu)  - Q_Ku(\xx_K;\mu)\Big] \right\rvert + \norm{\Phi_i}_{L^{\infty}(\Omega)}\norm{u}_{H^2(\Omega)}H^2.
  \label{test3}
\end{equation}

With the triangle inequality, the first term in \eqref{test3} becomes

\begin{align}
  \left\rvert \underset{K\in \mathcal{M}_H}{\sum} \Phi_i(\xx_K) |K| \Big[Q_Ku(\overline{\xx}_K;\mu)  - Q_Ku(\xx_K;\mu)\Big] \right\rvert &\lesssim \left\rvert \underset{K\in \mathcal{M}_H}{\sum} \Big[\Phi_i(\xx_G)+(\Phi_i(\xx_K)-\Phi_i(\xx_G))\Big] |K| \Big[Q_Ku(\overline{\xx}_K;\mu)  - Q_Ku(\xx_K;\mu)\Big] \right\rvert, \nonumber \\
  &\lesssim \left\rvert \underset{K\in \mathcal{M}_H}{\sum} \Phi_i(\xx_G) |K| \Big[Q_Ku(\overline{\xx}_K;\mu)  - Q_Ku(\xx_K;\mu)\Big] \right\rvert \nonumber \\& +  \norm{\nabla \Phi_i}_{L^{\infty}(\Omega)}  \underset{K\in \mathcal{M}_H}{\sum} H_K  |K|  \left\rvert Q_Ku(\overline{\xx}_K;\mu)  - Q_Ku(\xx_K;\mu) \right\rvert  \textrm{ since diam($U_G$) $\leq \mu_G H_K$ }.  \label{test4}
\end{align}

Using the decomposition of the mesh in patches $U_G$ and with the definition of $Q_K$, the first term of \eqref{test4} gives
\begin{align}
  \left\rvert \underset{K\in \mathcal{M}_H}{\sum} \Phi_i(\xx_G) |K| \Big[Q_Ku(\overline{\xx}_K;\mu)  - Q_Ku(\xx_K;\mu)\Big] \right\rvert &\leq \left\rvert  \underset{G\in \mathfrak{G}}{\sum} \underset{K\in G}{\sum} \Phi_i(\xx_G) \frac{|K|}{|B_K|} \int_{B_K}D^1u(\yy)\cdot \eee_K \ d\yy \right\rvert, \nonumber \\
  &\leq \underset{G\in \mathfrak{G}}{\sum} \norm{\Phi_i}_{L^{\infty}(G)}\left\rvert \underset{K\in G}{\sum} \Big( \frac{1}{|B_K|} \int_{B_K}D^1u(\yy) \ d\yy \Big) |K|\ \eee_K  \right\rvert.\label{test8}
\end{align}

Using the definition of $Q_K$ \eqref{defQ}, the second term in \eqref{test4} yields
\begin{align}
  \norm{\nabla \Phi_i}_{L^{\infty}(\Omega)}  \underset{K\in \mathcal{M}_H}{\sum} H_K |K|  \left\rvert Q_Ku(\overline{\xx}_K;\mu)  - Q_Ku(\xx_K;\mu) \right\rvert &=  \norm{\nabla \Phi_i}_{L^{\infty}(\Omega)} \underset{K\in \mathcal{M}_H}{\sum} H_K  \frac{|K|}{|B_K|}  \left\rvert \int_{B_K} D^1u(\yy)\cdot \eee_K \ d\yy \right\rvert, \nonumber \\
  &\lesssim \norm{\nabla \Phi_i}_{L^{\infty}(\Omega)} \underset{K\in \mathcal{M}_H}{\sum} H_K^2 \norm{\nabla u}_{L^1(B_K)}, \textrm{ since $|B_K|\geq \theta_H^{-1} |K|$ \eqref{reg}}, \nonumber \\
  & \leq H^2\norm{\nabla \Phi_i}_{L^{\infty}(\Omega)} \norm{\nabla u}_{L^1(\Omega)}.\label{test6}
\end{align}

Thus \eqref{test4} becomes
\begin{align}
 \left\rvert \underset{K\in \mathcal{M}_H}{\sum} \Phi_i(\xx_K) |K| \Big[ Q_Ku(\overline{\xx}_K;\mu)  - Q_Ku(\xx_K;\mu) \Big] \right\rvert &\lesssim \underset{G\in \mathfrak{G}}{\sum} \norm{\Phi_i}_{L^{\infty}(G)}\left\rvert \underset{K\in G}{\sum} \Big( \frac{1}{|B_K|} \int_{B_K}D^1u(\yy) \ d\yy \Big) |K|\ \eee_K  \right\rvert \nonumber\\ &+ H^2\norm{\nabla \Phi_i}_{L^{\infty}(\Omega)} \norm{\nabla u}_{L^1(\Omega)}. \label{test7}
\end{align}

Now, the Lemma 7.6. in ~\cite{supconv2} is going to be used three times on the first term the right hand side of \eqref{test7}. This lemma reads: \\
Let $U,\ V$ and $O$ be open sets of $\mathbb{R}^d$ such that, for all $(\xx,\mathbf{y})\in U \times V, \ [\xx,\mathbf{y}] \subset O$. There exists $C$ only depending on $d$ such that, for all $\Phi \in W^{1,1}(O)$,
\begin{equation}
\left\rvert \frac{1}{|U|}\int_U \Phi(\xx) \ d\xx - \frac{1}{|V|}\int_V \Phi(\xx) \ d\xx \right\rvert \leq C \frac{\textrm{diam}(O)^{d+1}}{|U||V|}\int_O |\nabla \Phi(\xx)|\ d\xx.
\end{equation}

\noindent We will use it successively with $[U,V,O]=[B_K,K,U_G]$, $[U,V,O]=[K,B_G,U_G]$, and $[U,V,O]=[B_G,U_G,U_G]$.\\

\noindent We use the triangle inequality on \eqref{test8},

\begin{align}
  \underset{G\in \mathfrak{G}}{\sum} \norm{\Phi_i}_{L^{\infty}(G)} \left| \underset{K\in G}{\sum} \Big( \frac{1}{|B_K|} \int_{B_K} D^1 u(\yy) \ d\yy \Big) |K| \ \eee_K  \right| &\leq  \underset{G\in \mathfrak{G}}{\sum} \norm{\Phi_i}_{L^{\infty}(G)} \Big| \underset{K\in G}{\sum} \Big( \left| \frac{1}{|B_K|} \int_{B_K}D^1u(\yy) \ d\yy - \ \frac{1}{|K|}  \int_{K} D^1 u(\mathbf{y};\mu)\ d\mathbf{y} \right| \nonumber \\
    &+ \ \left| \frac{1}{|K|}  \int_{K} D^1 u(\mathbf{y};\mu) \ d\mathbf{y} \ - \ \frac{1}{|B_G|}  \int_{B_G} D^1 u(\mathbf{y};\mu) \ d\mathbf{y} \right| \nonumber  \\
    &+  \left| \frac{1}{|B_G|}  \int_{B_G} D^1 u(\mathbf{y};\mu)\ d\mathbf{y} \ - \ \frac{1}{|U_G|}  \int_{U_G} D^1 u(\mathbf{y};\mu)\ d\mathbf{y} \right|  \nonumber \\
    &+     \frac{1}{|U_G|}  \int_{U_G} D^1 u(\mathbf{y};\mu)\ d\mathbf{y} \Big)\ |K| \ \eee_K \Big| . \label{test10}
    \end{align}
    and we get
    
    \begin{align}
  \underset{G\in \mathfrak{G}}{\sum} \norm{\Phi_i}_{L^{\infty}(G)} \left| \underset{K\in G}{\sum} \Big( \frac{1}{|B_K|} \int_{B_K} D^1 u(\yy) \ d\yy \Big) |K| \ \eee_K  \right| &\lesssim  \underset{G\in \mathfrak{G}}{\sum} \norm{\Phi_i}_{L^{\infty}(G)} \Big| \underset{K\in G}{\sum} \Big( \norm{u}_{W^{2,1}(U_G)}\textrm{diam}(U_G)^d \Big[  \frac{\textrm{diam}(U_G)}{|B_K||K|} \nonumber \\
    &+ \frac{\textrm{diam}(U_G)}{|B_G||K|}+\frac{\textrm{diam}(U_G)}{|U_G||B_G|} \Big] +  \frac{1}{|U_G|}  \int_{U_G} D^1 u(\mathbf{y};\mu) \ d\mathbf{y} \Big) \ |K| \ \eee_K \Big| .  \label{test11}
    \end{align}

\noindent With the regularity factor $\theta_{H}$ (see the previous definition of a polytopal mesh \eqref{reg}), $|K|\leq |B(0,1)|H_K^d \lesssim |B_K|\theta_H^{d}. $ Since $\textrm{Card}(G)$ is bounded by $\mu_G$, $\textrm{diam}(U_G)\leq \mu_G H_K$. Thus, 
$\textrm{diam}(U_G)^d\leq \mu_G^d H_K^d,$ and $\frac{\textrm{diam}(U_G)}{|B_K|}\leq C$, $\ |B_G|\geq \mu_G^{-d} \textrm{diam}(U_G)^d,\ |B_G|\gtrsim \mu_G^{-d}H_K^d\gtrsim \mu_G^{-d}|K|,$ and $|U_G|\geq \textrm{diam}(U_G)^d.$\\
Therefore \eqref{test11} becomes 
    \begin{align}
  \underset{G\in \mathfrak{G}}{\sum} \norm{\Phi_i}_{L^{\infty}(G)} \left| \underset{K\in G}{\sum} \Big( \frac{1}{|B_K|} \int_{B_K} D^1 u(\yy) \ d\yy \Big) |K| \eee_K  \right| &\lesssim  \underset{G\in \mathfrak{G}}{\sum} \norm{\Phi_i}_{L^{\infty}(G)} \Big| \underset{K\in G}{\sum} \Big( \norm{u}_{W^{2,1}(U_G)}\frac{\textrm{diam}(U_G)}{|K|} \nonumber \\
    &+  \frac{1}{|U_G|}  \int_{U_G} D^1 u(\mathbf{y};\mu) \ d\mathbf{y} \Big) \ |K| \ \eee_K \Big|.  \label{test12}
    \end{align}

Since $\textrm{diam}(U_G)\leq  \mu_G H_K$ and $|\eee_K|\leq H_K$,
\begin{align}
  \underset{G\in \mathfrak{G}}{\sum} \norm{\Phi_i}_{L^{\infty}(G)} \left| \underset{K\in G}{\sum} \Big( \frac{1}{|B_K|} \int_{B_K} D^1 u(\yy) \ d\yy \Big) |K|\ \eee_K  \right| &\lesssim  \underset{G\in \mathfrak{G}}{\sum} \norm{\Phi_i}_{L^{\infty}(G)} \Big[ \underset{K\in G}{\sum} H_K^2 \norm{u}_{W^{2,1}(U_G)} \nonumber \\
    &+ \left\rvert \frac{1}{|U_G|} \underset{K\in G}{\sum} \int_{U_G} D^1 u(\mathbf{y};\mu)\ d\mathbf{y}  |K| \eee_K  \right\rvert \Big] . \label{test13}
\end{align}

Then,
\begin{align}
 \underset{G\in \mathfrak{G}}{\sum} \norm{\Phi_i}_{L^{\infty}(G)} \left| \underset{K\in G}{\sum} \Big( \frac{1}{|B_K|} \int_{B_K} D^1 u(\yy) \ d\yy \Big) |K| \eee_K  \right|  &\lesssim  \underset{G\in \mathfrak{G}}{\sum} \norm{\Phi_i}_{L^{\infty}(G)} \underset{K\in G}{\sum} H_K^2 \norm{u}_{W^{2,1}(U_G)} \nonumber \\
    &+   \underset{G\in \mathfrak{G}}{\sum} \norm{\Phi_i}_{L^{\infty}(G)}\left\rvert \frac{1}{|U_G|} \underset{K\in G}{\sum}  |K|\ \eee_K  \right\rvert \left\rvert \int_{U_G} D^1 u(\mathbf{y};\mu)\ d\mathbf{y}  \right\rvert, \label{test14}
\end{align}

which implies, since $\textrm{Card}(G)\leq \mu_G$,
\begin{align}
\underset{G\in \mathfrak{G}}{\sum} \norm{\Phi_i}_{L^{\infty}(G)} \left| \underset{K\in G}{\sum} \Big( \frac{1}{|B_K|} \int_{B_K} D^1 u(\yy) \ d\yy \Big) |K| \eee_K  \right|  &\lesssim  \underset{G\in \mathfrak{G}}{\sum} \norm{\Phi_i}_{L^{\infty}(G)} H^2 \norm{u}_{W^{2,1}(U_G)} \nonumber \\
    &+   \underset{G\in \mathfrak{G}}{\sum} \norm{\Phi_i}_{L^{\infty}(G)}\left\rvert \frac{1}{|U_G|} \underset{K\in G}{\sum}  |K| \ \eee_K  \right\rvert \norm{u}_{W^{1,1}(U_G)}. \label{test15}
\end{align}

and finally,
\begin{equation}
 \underset{G\in \mathfrak{G}}{\sum} \norm{\Phi_i}_{L^{\infty}(G)}\left\rvert \underset{K\in G}{\sum} \frac{1}{|B_K|} \int_{B_K}D^1u(\yy) \ d\yy |K|\ \eee_K  \right\rvert \leq  \norm{\Phi_i}_{L^{\infty}(\Omega)} \norm{u}_{W^{2,1}(\Omega)} H^2 
    +   \norm{\Phi_i}_{L^{\infty}(\Omega)} \underset{G \in \mathfrak{G}}{\max\ } \norm{u}_{W^{1,1}(\Omega)} \left\rvert \frac{1}{|U_G|} \underset{K\in G}{\sum}  |K| \ \eee_K  \right\rvert. \label{test16}
\end{equation}

This results  using \eqref{test2}, \eqref{inter22}, \eqref{test4}, \eqref{test6}, and \eqref{test16} in
\begin{equation}
 \left\rvert \int_{\Omega} (\Pi_1^H u(\mu)-\Pi_0^Hu(\mu)) \cdot \Pi_0^H\Phi_i \ d\xx \right\rvert \lesssim (\norm{\Phi_i}_{W^{1,\infty}(\Omega)}\norm{u}_{W^{2,1}(\Omega)}+\norm{u}_{H^2(\Omega)}\norm{\Phi_i}_{L^{\infty}(\Omega)})H^2+(\norm{\Phi_i}_{L^{\infty}(\Omega)}\norm{u}_{W^{1,1}(\Omega)})e_{\mathfrak{G}}. 
\end{equation}
If $e_{\mathfrak{G}}=\underset{G \in \mathfrak{G}}{\max\ }  \left\rvert \bigg(\frac{1}{|U_G|}  \underset{K \in G}{\sum} |K|\ \eee_K\bigg)\right\rvert$ is in $\mathcal{O}(H^2)$ then the estimate of $ \left\rvert \int_{\Omega} (\Pi_1^H u(\mu)-\Pi_0^Hu(\mu)) \cdot \Pi_0^H\Phi_i \ d\xx \right\rvert$ is in $\mathcal{O}(H^2).$
This concludes the proof since the rest is similar to the one of Theorem \Rref{th11}. Note that for the estimate of $T_{3,2}$ \eqref{term2bis}, the equation \eqref{superconvhmm2} from the Theorem of super-convergence with local grouping is used instead of \eqref{superconvhmm}.
\end{proof}
\section{Some details on the implementation and numerical results}
We consider two simple cases in 2D for the numerical results based with the TPFA scheme. Both results are computed on the unit square. We use an harmonic averaging of the diffusion coefficient(~\cite{jd} section 5.3). Our variable parameter is $\mu\in \mathbb{R}^4=(\mu_1,\mu_2,\mu_3,\mu_4)$. For both cases, the size of mesh $h$ is defined as the maximum length of the edges. The diffusion coefficient we consider here is $A(\mu)=(2\mu_1+\mu_2 \sin(x+y)\cos(xy))$ and $f=(\mu_3(1-y)+\mu_4x(1-x))$.
We choose random coefficients for the snapshots with $N=5$ and our solution is defined with $\mu_1=0.99,\ \mu_2=0.8,\ \mu_3=0.2,\ \mu_4=0.78$. For the exact solution, we consider the TPFA solution on a finer mesh (Figures \Rref{fig:rectangle}, \Rref{fig:triangle}). For the computation of the norm, we use the discrete semi-norm as in the remark of the section \Rref{sect2} \eqref{discretenorme}. NIRB results are compared to the classical finite volume error  (Figures \Rref{fig:resurect}, \Rref{fig:resutriangle}). We measure the following relative error
\begin{equation}
  \frac{\norm{u(\mu) - u_{Hh}^N(\mu)}_{\ttens,2}}{\norm{u(\mu)}_{\ttens,2}}.
  \end{equation}

%\vspace{2cm}
%\begin{figure}[H]
%\begin{center}
%\begin{pspicture}(0,0)(4,4)
%\pspolygon[opacity=0.1,fillstyle=solid,fillcolor=yellow](0,0)(5,0)(5,5)(0,5)
%\psline[linestyle=dashed](2.5,0)(2.5,2.5)
%\psline(2.5,0)(2.5,2.5)
%\psline(0,2.5)(2.5,2.5)
%\psline(2.5,5)(2.5,2.5)
%\psline(2.5,2.5)(5,0)
%\uput[d](1.7,1.45){$\mu_1$}
%\uput[d](3.6,3.45){$\mu_2$}
%\uput[d](1.2,3.45){$\mu_3$}
%\uput[d](3.2,1.45){$\mu_4$}
%\end{pspicture}
%\end{center}
%\caption{Domain decomposition}
%\end{figure}

\paragraph{Uniform grid}

The first case presents results on a rectangular uniform grid where $\xx_K $ is the center of mass of the cell.

\vspace{-0.8cm}
\begin{figure}[H]
  \centering
  \includegraphics[scale=0.13]{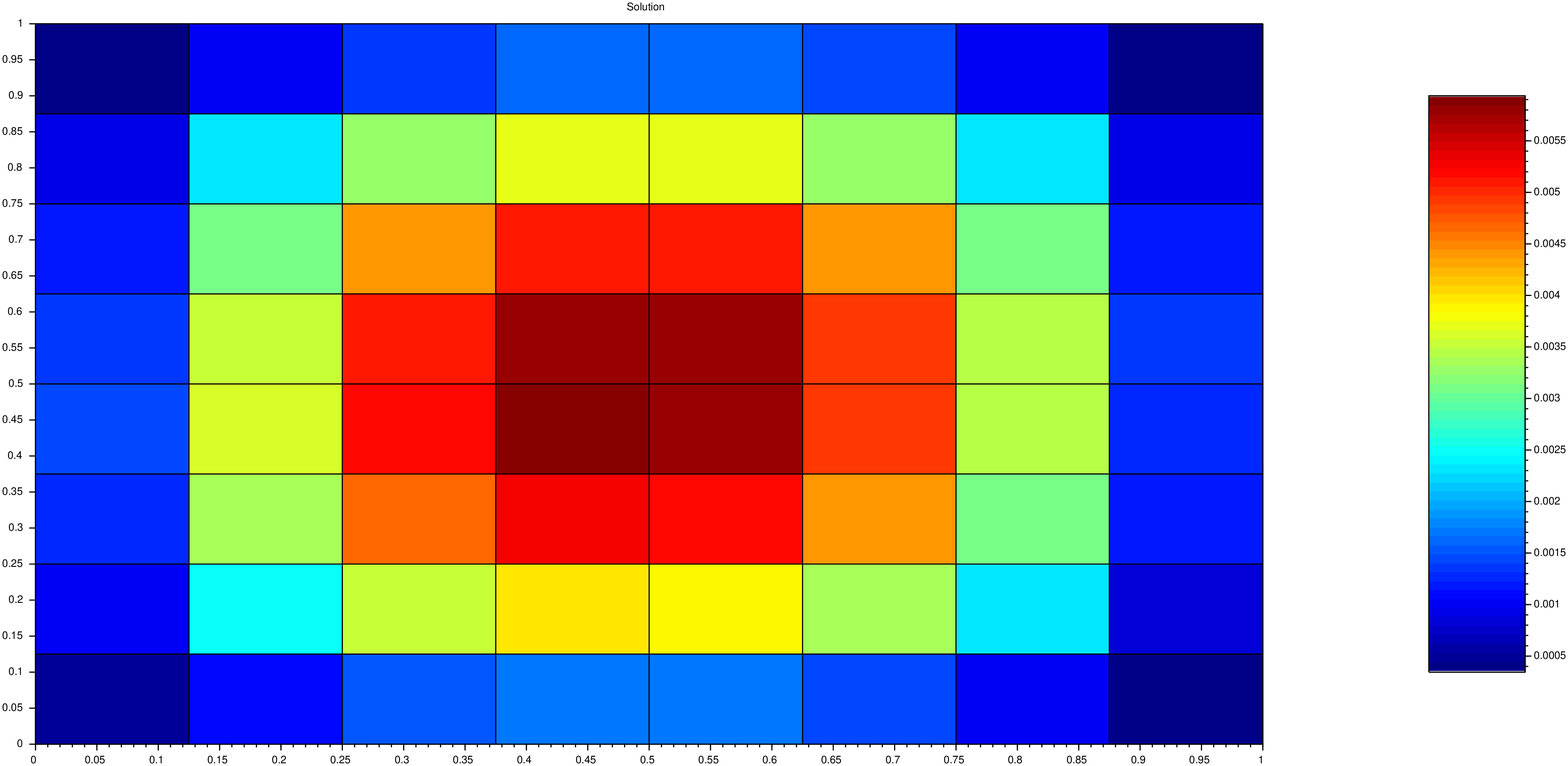}\qquad
\includegraphics[scale=1.5]{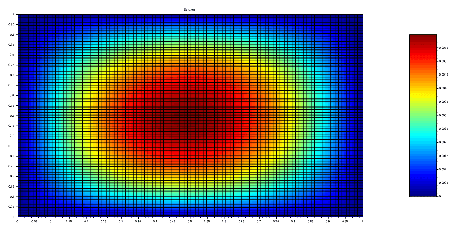}
\caption{coarse and fine solution with the uniform grid}\label{fig:rectangle}
\end{figure}

\begin{figure}[H]
    \centering
    \includegraphics[scale=0.3]{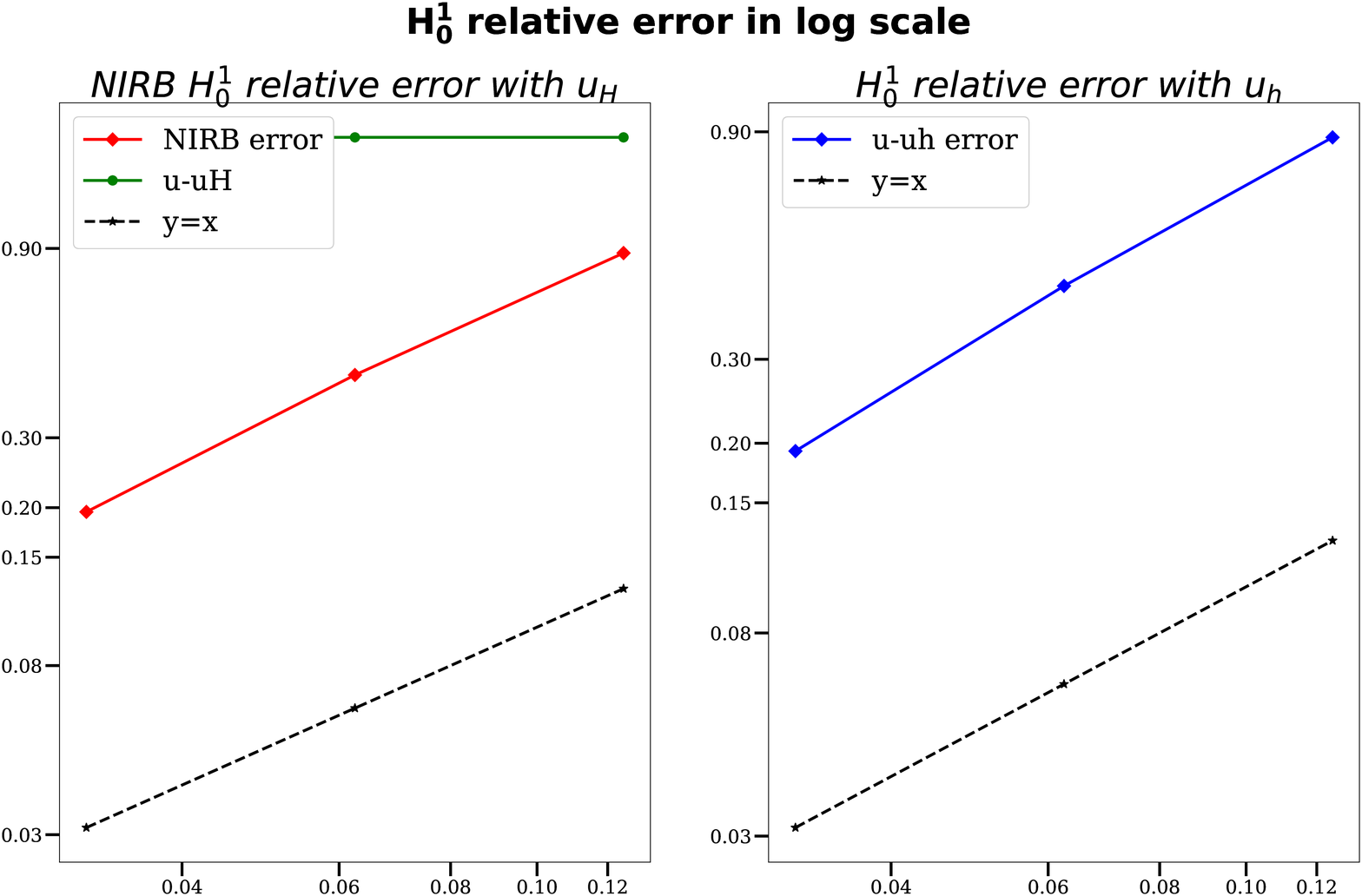}
   \caption{Numerical result on the uniform grid}
  \label{fig:resurect}
\end{figure}

\paragraph{Triangular mesh}

The second case is defined on a triangular mesh where $\xx_K $ are the circumcenter of the cells, such that $e_{\mathfrak{G}}$ is in $\mathcal{O}(H^2)$. \\

\vspace{-0.5cm}
\begin{figure}[H]
  \centering
\includegraphics[scale=0.13]{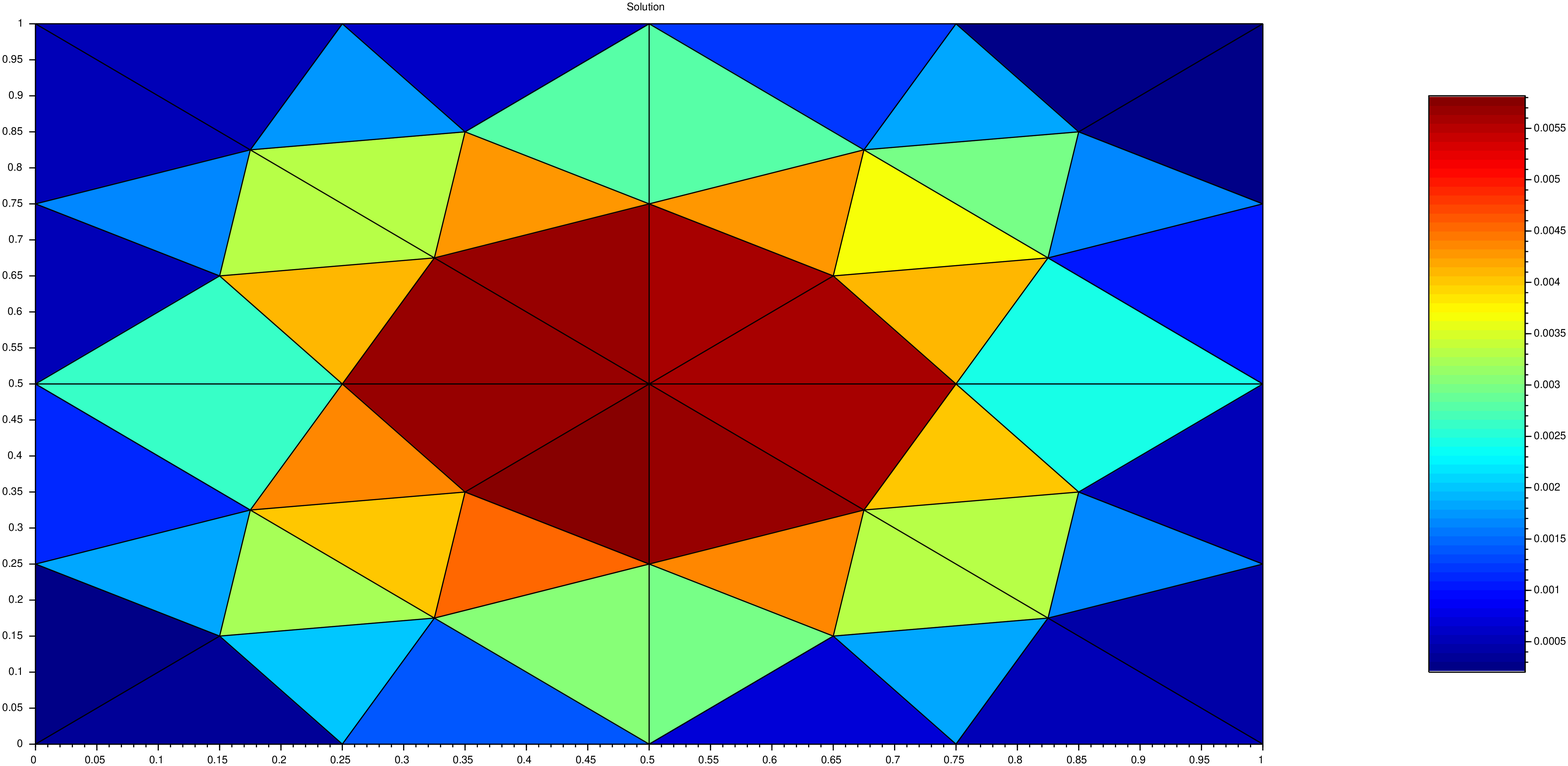}\qquad
\includegraphics[scale=1.5]{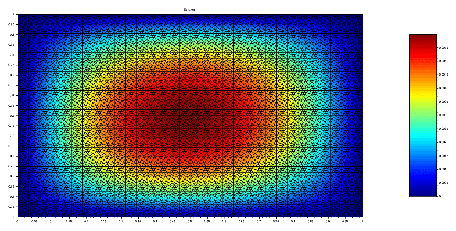}
\caption{coarse and fine solution with the triangular mesh}\label{fig:triangle}
\end{figure}

\begin{figure}[H]
    \centering
    \includegraphics[scale=0.3]{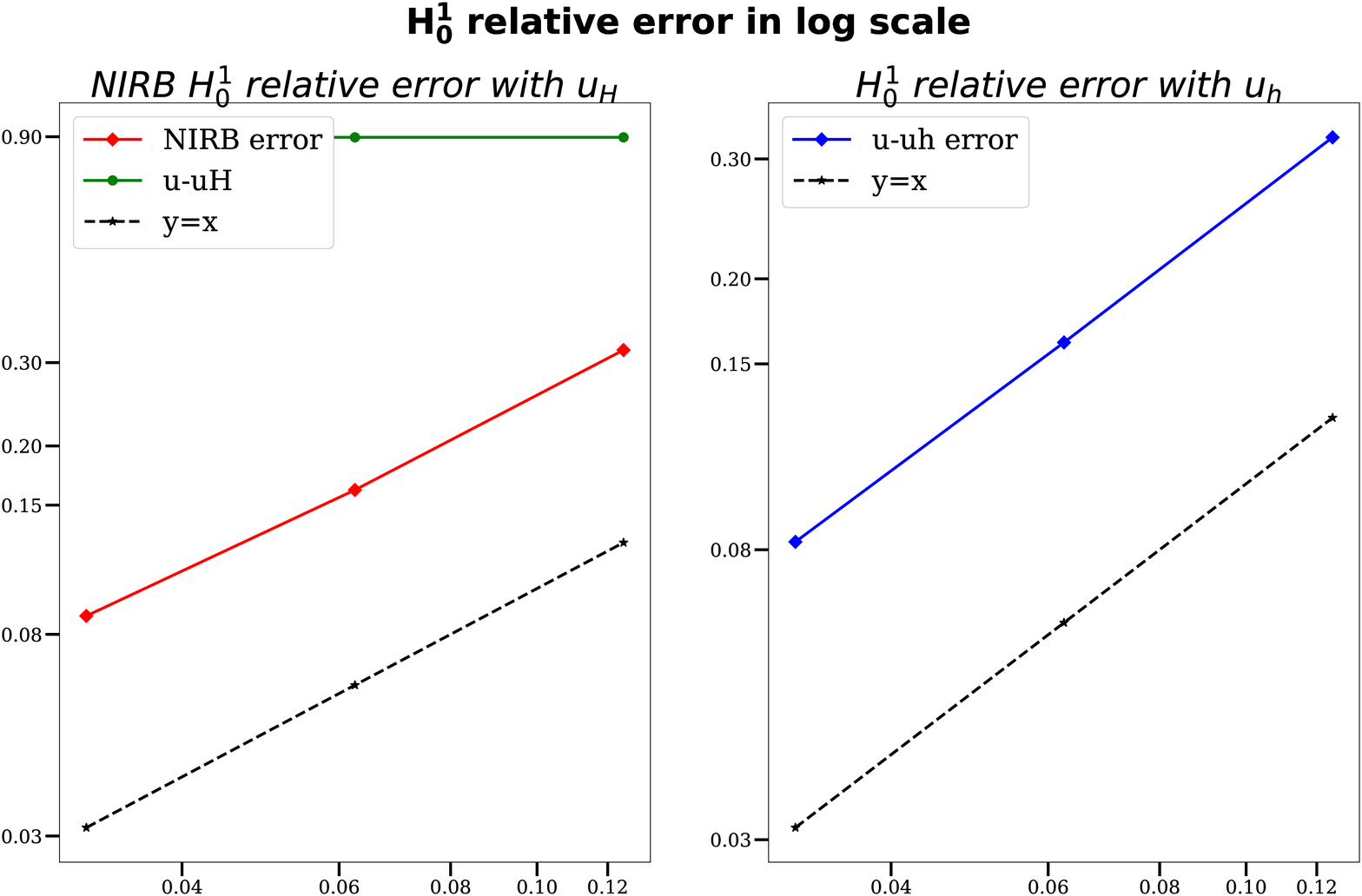}
    \caption{Numerical result on the triangular mesh}
    \label{fig:resutriangle}
  
\end{figure}

\paragraph{Discussion on the implementation}

We implemented the TPFA scheme on Scilab and retrieved several solutions for the NIRB algorithm on Python to highlight the black box side of the solver. 

\begin{itemize}
\item Implementation of TPFA\\
  The TPFA on $\mathcal{T}_h$ reads:
  Find $u_h=(u_K)_{K \in \mm}$ such that:\\
  \begin{equation}
    \forall K\in \mm_h,\underset{\sigma \in \mathcal{F}_K\cap \mathcal{F}_{int}}{\sum} \tau_{\sigma}(u_K-u_L) + \underset{\sigma \in \mathcal{F}_K\cap \mathcal{F}_{ext}}{\sum} \tau_{\sigma} u_K= \int_K f(\xx) d\xx,
     \end{equation}
    where the harmonic average $\tau_{\sigma}=|\sigma|\frac{A(\xx_L;\mu)A(\xx_K;\mu)}{A(\xx_L;\mu)\times d_{L,\sigma}+ A(\xx_K;\mu)\times d_{K,\sigma}}$ on $\mathcal{F}_{int}$, and $\tau_{\sigma}=|\sigma|\frac{A(\xx_K;\mu)}{d_{K,\sigma}}$ on $\mathcal{F}_{ext}$.
   
   To assemble the matrices $A$ of the TPFA scheme, we iterate on each edge, and we add the harmonic average $\tau_{\sigma}$ on each cell, and for $b$ we add the term $|D_{K,\sigma}|\times f(x_K)$.

 \vspace{1cm}
  \item Time execution (min,sec)

    \begin{tabular}{ |c ||c| c| }
      \hline
      & NIRB Online & FV solver \\
      \hline
      uniform grid & 00:06 & 01:48 \\
      \hline
      triangular mesh & 00:05 & 01:15 \\
      \hline
    \end{tabular}
    
\end{itemize}
 \vspace{1cm}
 \begin{rmrk}
   Note that for the discontinuous diffusion coefficient $A$, with the TPFA scheme, we recovered numerically the same estimate as in the Lipschitz continuous case, when we use the harmonic mean even if the proof no longer works.
   
\end{rmrk}

 \paragraph{Acknowledgements}
 This work is supported by the FUI MOR\_DICUS. We would like to give special thanks to Nora Aïssiouene at the LJLL for her precious help.%\tnoteref{mytitlenote}}

\newpage

%Some precision:
%If $\quadru$ is a polytopal mesh of $\Omega$, we define the space of cell and face unknowns by \begin{equation*}
 %    X_\mathcal{T} = \{{v = ((v_K)_{K\in\mm},(v_\sigma)_{\sigma \in \ff} ) : v_K \in \mathbb{R},\  v_{\sigma} \in \mathbb{R}}\},
%\end{equation*}
%and the subspace of vectors with a zero value on the boundary by
%\begin{equation*}
 %   X_{\ttens, 0}  = \{{v \in X_{\mathcal{T}} : v_{\sigma} = 0 \textrm{ for all }\sigma \in \ff_{ext}}\}.
%\end{equation*}
%The function reconstruction $\Pi_\ttens : X_\ttens \to L^{\infty}(\Omega)$, and gradient reconstruction $\nabla \ttens : X_\ttens \to L^{\infty}(\Omega)^d$ are defined by
%\begin{itemize}
%     \item $\forall v\in X_\ttens, \ \forall K \in \mm, \ \textrm{ for } \textrm{ a.e. }\xx \in K,\  \Pi_{\ttens}v(\xx)=v_K.$
%\item  $\forall v \in X_{\ttens}, \forall K \in \mm, \forall \textrm{ a.e. } \xx \in K,\ \nabla_\ttens v(\xx)=\nabla_K v=\underset{\sigma \in \ff_K}{\sum} |\sigma| v_{\sigma}\nn_{K,\sigma}.$
%\end{itemize}
%For $p \in [1, +\infty)$ a discrete $W^{1,p}$ semi-norm on $X_\ttens$ is defined by
%\begin{equation}
 %   \forall v \in \mathcal{T},\ |v|_{\ttens,p}^p= \underset{K \in \mm}{\sum} \underset{\sigma \in \ff_K}{\sum} |\sigma| d_{K,\sigma}|\frac{v_{\sigma}-v_K}{d_{K,\sigma}}|^p.
%\end{equation}
%In the next paragraphs, we consider $p=2$.\\

%\begin{enumerate}[(1)]
%\item Group the authors per affiliation.
%\item Use footnotes to indicate the affiliations.
%\end{enumerate}

\bibliography{mybibfile}

\begin{thebibliography}{10}

\bibitem{barrault2004empirical}
M.~Barrault, C.~Nguyen, A.~Patera, and Y.~Maday.
\newblock {An `empirical interpolation' method: application to efficient
  reduced-basis discretization of partial differential equations}.
\newblock {\em {Comptes rendus de l'Acad{\'e}mie des sciences. S{\'e}rie I,
  Math{\'e}matique}}, 339-9:667--672, 2004.

\bibitem{boyer}
F.~Boyer.
\newblock An introduction to finite volume methods for diffusion problems.
\newblock {\em French-Mexican Meeting on Industrial and Applied Mathematics
  Villahermosa, Mexico}, November 25-29 2013.

\bibitem{FE}
S.~Brenner and R.~Scott.
\newblock {\em The mathematical theory of finite element methods}, volume~15.
\newblock Springer Science \& Business Media, 2007.

\bibitem{hmfd}
F.~Brezzi, K.~Lipnikov, and M.~Shashkov.
\newblock Convergence of the mimetic finite difference method for diffusion
  problems on polyhedral meshes.
\newblock {\em SIAM Journal on Numerical Analysis}, 43(5):1872--1896, 2005.

\bibitem{greedy}
A.~Buffa, Y.~Maday, A.~T. Patera, C.~Prud’homme, and G.~Turinici.
\newblock A priori convergence of the greedy algorithm for the parametrized
  reduced basis method.
\newblock {\em ESAIM: Mathematical Modelling and Numerical
  Analysis-Modélisation Mathématique et Analyse Numérique}, 46(3):595 --
  603, 2012.

\bibitem{Casenave2014}
Fabien Casenave, Alexandre Ern, and Tony Lelièvre.
\newblock A nonintrusive reduced basis method applied to aeroacoustic
  simulations.
\newblock {\em Advances in Computational Mathematics}, 41(5):961–986, Jun
  2014.

\bibitem{NIRB1}
R.~Chakir.
\newblock {\em Contribution à l'analyse numérique de quelques problèmes en
  chimie quantique et mécanique}.
\newblock PhD thesis, 2009.
\newblock Doctoral dissertation.

\bibitem{NIRB3}
R.~Chakir, P.~Joly, Y.~Maday, and P.~Parnaudeau.
\newblock A non intrusive reduced basis method: application to computational
  fluid dynamics.
\newblock 2013, September.

\bibitem{NIRB2}
R.~Chakir, Y.~Maday, and P.~Parnaudeau.
\newblock A non-intrusive reduced basis approach for parametrized heat transfer
  problems.
\newblock {\em Journal of Computational Physics}, 376:617--633, 2019.

\bibitem{cohen}
A.~Cohen and R.~DeVore.
\newblock Approximation of high-dimensional parametric pdes.
\newblock {\em arXiv preprint arXiv:1502.06797}, 2015.

\bibitem{hmm2}
L.~B. da~Veiga, K.~Lipnikov, and G.~Manzini.
\newblock {\em The mimetic finite difference method for elliptic problems},
  volume~11.
\newblock Springer, 2014.

\bibitem{supconv1}
D.~A. Di~Pietro, A.~Ern, and S.~Lemaire.
\newblock An arbitrary-order and compact-stencil discretization of diffusion on
  general meshes based on local reconstruction operators.
\newblock {\em Computational Methods in Applied Mathematics}, 14(4):461--472,
  2014.

\bibitem{tpfascheme}
J.~Droniou.
\newblock Finite volume schemes for diffusion equations: Introduction to and
  review of modern methods.
\newblock {\em Mathematical Models and Methods in Applied Sciences}, 24
  (8):1575--1619, 2014.

\bibitem{mfv}
J.~Droniou and R.~Eymard.
\newblock A mixed finite volume scheme for anisotropic diffusion problems on
  any grid.
\newblock {\em Numerische Mathematik}, 105(1):35--71, 2006.

\bibitem{hmm3}
J.~Droniou, R.~Eymard, T.~Gallouet, and R.~Herbin.
\newblock Gradient schemes: a generic framework for the discretisation of
  linear, nonlinear and nonlocal elliptic and parabolic equations.
\newblock {\em Mathematical Models and Methods in Applied Sciences},
  23(13):2395--2432, 2013.

\bibitem{cindy}
J.~Droniou, R.~Eymard, T.~Gallouët, C.~Guichard, and R.~Herbin.
\newblock The gradient discretisation method.
\newblock {\em Springer}, 82, 2018.

\bibitem{jd}
J.~Droniou, R.~Eymard, T.~Gallouët, and R.~Herbin.
\newblock A unified approach to mimetic finite difference, hybrid finite volume
  and mixed finite volume methods.
\newblock {\em Mathematical Models and Methods in Applied Sciences},
  20(02):265--295, 2010.

\bibitem{supconv2}
J.~Droniou and N.~Nataraj.
\newblock Improved $l^2$ estimate for gradient schemes and super-convergence of
  the tpfa finite volume scheme.
\newblock 3:1254--1293, 2017.

\bibitem{hfv}
R.~Eymard, T.~Gallou{\"e}t, and R~Herbin.
\newblock Discretization schemes for linear diffusion operators on general
  non-conforming meshes.
\newblock {\em Finite volumes for complex applications V}, pages 375--382,
  2008.

\bibitem{haasdonk2008reduced}
B.~Haasdonk and M.~Ohlberger.
\newblock Reduced basis method for explicit finite volume approximations of
  nonlinear conservation laws.
\newblock In {\em Proc. 12th International Conference on Hyperbolic Problems:
  Theory, Numerics, Application}. Citeseer, 2008.

\bibitem{hesthaven2016certified}
Jan~S Hesthaven, Gianluigi Rozza, Benjamin Stamm, et~al.
\newblock {\em Certified reduced basis methods for parametrized partial
  differential equations}.
\newblock Springer, 2016.

\bibitem{iliev2013two}
O.~Iliev, Y.~Maday, and T.~Nagapetyan.
\newblock {\em A Two-grid Infinite-volume/reduced Basis Scheme for the
  Approximation of the Solution of Parameter Dependent PDE with Applications
  the AFFFF Devices}.
\newblock Citeseer, 2013.

\bibitem{kolmo}
A.~Kolmogoroff.
\newblock Über die beste annaherung von funktionen einer gegebenen
  funktionenklasse.
\newblock {\em Annals of Mathematics}, pages 107--110, 1936.

\bibitem{madaychakir}
Y.~Maday and R.~Chakir.
\newblock A two-grid finite-element/reduced basis scheme for the approximation
  of the solution of parametric dependent p.d.e.
\newblock 2009.

\bibitem{rb}
A.~Quarteroni, A.~Manzoni, and F.~Negri.
\newblock Reduced basis methods for partial differential equations: an
  introduction.
\newblock {\em Springer}, 92, 2015.

\bibitem{nummod}
A.~Quarteroni and S.~Quarteroni.
\newblock {\em Numerical models for differential problems}, volume~2.
\newblock Springer, 2009.

\bibitem{TPFABR}
R.~Sanchez.
\newblock Application des techniques de bases réduites à la simulation des
  écoulements en milieux poreux.
\newblock {\em Université Paris-Saclay - CentraleSupélec}, 2017.

\bibitem{stabile2017pod}
G.~Stabile, S.~Hijazi, A.~Mola, S.~Lorenzi, and G.~Rozza.
\newblock Pod-galerkin reduced order methods for cfd using finite volume
  discretisation: vortex shedding around a circular cylinder.
\newblock {\em Communications in Applied and Industrial Mathematics},
  8(1):210--236, 2017.

\bibitem{veroy2003posteriori}
K.~Veroy, C.~Prud'Homme, and A.~T. Patera.
\newblock {Reduced-basis approximation of the viscous Burgers equation:
  rigorous a posteriori error bounds}.
\newblock {\em {Comptes rendus de l'Acad{\'e}mie des sciences. S{\'e}rie I,
  Math{\'e}matique}}, 337(9):619--624, November 2003.

\end{thebibliography}
\bibliographystyle{plain}
%\bibliography{mybibfile}
%\begin{thebibliography}{1}
%    \bibitem{latexpratique} Christian \textsc{Rolland}. \emph{\LaTeX{} par la pratique}. O'Reilly, 1999.
%\end{thebibliography}

\end{document}